\documentclass[12pt,showkeys]{amsart}
\def\makeheadbox{{%
\hbox to0pt{\vbox{\baselineskip=10dd\hrule\hbox
to\hsize{\vrule\kern3pt\vbox{\kern3pt
\hbox{\bfseries Draft for discussion }
\hbox{Date of this version: 13.04.21}
\kern3pt}\hfil\kern3pt\vrule}\hrule}%
\hss}}}

\usepackage{color}
\usepackage{graphicx}
\usepackage{subfigure}
\usepackage{amssymb}
\usepackage{amsmath}
\usepackage{multirow}
\usepackage{enumerate}
\usepackage{epstopdf}
\usepackage{cite,stmaryrd,txfonts}
\usepackage{tikz}
\usetikzlibrary{snakes}

\usepackage{epic}
\usepackage{empheq}
\usepackage{cases}

\def\cequiv{\raisebox{-1.5mm}{$\;\stackrel{\raisebox{-3.9mm}{=}}{{\sim}}\;$}}

\newtheorem{theorem}{Theorem}[section]
\newtheorem{remark}[theorem]{Remark}\newtheorem{proposition}[theorem]{Proposition}
\newtheorem{lemma}[theorem]{Lemma}

\newtheorem{definition}[theorem]{Definition}

\newcounter{mnote}
\let\oldmarginpar\marginpar
\renewcommand\marginpar[1]{\-\oldmarginpar[\raggedleft\footnotesize #1]
  {\raggedright\footnotesize #1}}
  
\usepackage{footnote}
\usepackage{booktabs}

\usepackage{rotating}

\numberwithin{equation}{section}

\usepackage{cite}

\usepackage[colorlinks,linkcolor=black,anchorcolor=black,citecolor=black]{hyperref}
\usepackage[shortlabels]{enumitem} 
\setlist[enumerate]{nosep}
\usepackage{color, soul}

\usepackage{stmaryrd}

\begin{document}
\title{Lowest-degree robust finite element scheme for a fourth-order elliptic singular perturbation problem on rectangular grids}
\author{Huilan Zeng}
\address{LSEC, ICMSEC, Academy of Mathematics and System Sciences, Chinese Academy of Sciences, Beijing 100190, China; School of Mathematical Sciences, University of Chinese Academy of Sciences, Beijing 100049, China}
\email{zhl@lsec.cc.ac.cn}
\author{Chen-Song Zhang}
\address{LSEC, ICMSEC and NCMIS, Academy of Mathematics and System Sciences, Chinese Academy of Sciences, Beijing 100190, China}
\email{zhangcs@lsec.cc.ac.cn}
\author{Shuo Zhang}
\address{LSEC, ICMSEC and NCMIS, Academy of Mathematics and System Sciences, Chinese Academy of Sciences, Beijing 100190, China}
\email{szhang@lsec.cc.ac.cn}
\thanks{Zeng and C.-S. Zhang were partially supported by the Science Challenge Project TZZT2019-B1.1, the National Science Foundation of China 11971472, and the Key Research Program of Frontier Sciences of CAS. S. Zhang is partially supported by National Natural Science Foundation,  11471026 and 11871465, China.}


\maketitle

\begin{abstract}
In this paper, a piecewise quadratic nonconforming finite element method on rectangular grids for a fourth-order elliptic singular perturbation problem is presented.  This proposed method is robustly convergent with respect to the perturbation parameter.  Numerical results are presented to verify the theoretical findings. 
~\\

The new method uses piecewise quadratic polynomials, and is of the lowest degree possible. Optimal order approximation property of the finite element space is proved by means of a locally-averaged interpolation operator newly constructed.  This interpolator, however, is not a projection. Indeed, we establish a general theory and show that no locally defined interpolation associated with the locally supported basis functions can be projective for the finite element space in use. Particularly, the general theory gives an answer to a long-standing open problem presented in [Demko, J. Approx. Theory, \textbf{43}(2):151--156, 1985].
\end{abstract}

\section{Introduction}
\label{intro}
Let $\Omega \subset \mathbb{R}^{2}$ be a simply-connected polygon, which can be covered by a rectangular subdivision, and $f \in L^{2}(\Omega)$. In this paper, we consider the fourth-order elliptic singular perturbation problem:
\begin{equation}\label{eq:model problem}
\left\{
\begin{array}{rl}
 \varepsilon^{2} \Delta^{2}u -\Delta u = f, & \mbox{ in } \ \Omega , \\
u = \frac{\partial u}{\partial \mathbf{n}} = 0, & \mbox{ on } \ \partial\Omega ,
\end{array}
\right.
\end{equation}
where $\frac{\partial u}{\partial \mathbf{n}}$ denotes the normal derivative along the boundary $\partial \Omega$, and $0< \varepsilon \leq 1$ is a real parameter. This equation models, for example, thin buckling plates with $u$ representing the displacement of the plate \cite{Frank1997}. 
~\\

There have been two main approaches to obtain a robust finite element scheme for the model problem~\eqref{eq:model problem}: (i)~designing a finite element discretization works for both fourth-order and second-order problems; (ii)~modifying the variational formulation of the model problem.  The first approach can be further divided into three categories. The first category uses conforming finite elements, such as the Argyris element or Hsieh-Clough-Toucher element. Because of their requirement of higher degree polynomials or complicated macroelement techniques, they are sometimes considered less friendly to practical applications. The second category employs $H^{2}$ nonconforming and $H^{1}$ conforming elements \cite{Chen;Chen2014,Chen;Chen;Qiao2013,Chen;Chen;Xiao2014,Guzman;Leykekhman;Neilan2012,Nilssen;Tai;Winther2001,Tai;Winther2006,Wang;Wu;Xie2013,Wang;Shi;Xu2007,2ndWang;Shi;Xu2007,Xie;Shi;Li2010,Zhang;Wang2008}. The third category involves $H^{1}$ nonconforming elements \cite{Chen;Liu;Qiao2010,Chen;Zhao;Shi2005,Wang;Wu;Xie2013}. As for the second approach based on modified variational formulations, the fourth-order and second-order bilinear formulations are handled separately. It is well-known that the triangular Morley element does not converge for the Poisson equation \cite{Wang2001,Nilssen;Tai;Winther2001}, and thus is not uniformly convergent with respect to $\varepsilon$.  Modified triangular Morley element methods are proposed in \cite{Wang;Xu;Hu2006} and \cite{Wang;Meng2007} in two and three dimensions, respectively. Another example of the second approach is the $C^{0}$ interior penalty discontinuous Galerkin (IPDG) method. It is devised for problem~\eqref{eq:model problem} in \cite{Brenner;Neilan2011} and reanalyzed for a layer-adapted mesh in \cite{Franz;Roos;Wachtel2014}. Moreover, by adopting the IPDG formulation to dispose the Laplace operator, two Morley--Wang--Xu element methods with penalty are presented in \cite{Wang;Huang;Tang;Zhou2018}. 
~\\

In this paper, we apply the reduced rectangular Morley (RRM) element to the singular perturbation problem~\eqref{eq:model problem}. Inspired by a discretized Stokes complex given in \cite{Zhang.S2016nm}, the RRM element is firstly proposed in \cite{Shuo.Zhang2020} for $H^{2}$ problems and later analyzed in \cite{Zeng.H;Zhang.C;Zhang.S2019} for $H^{1}$ problems. Basically, it is an optimal quadratic element for both $H^{2}$ and $H^{1}$ problems, and in turn optimal for \eqref{eq:model problem}. More precisely, this method is convergent with a rate of $\mathcal{O}(h)$ order in the energy norm with respect to the regularity of the exact solution. Moreover, when $\varepsilon$ is very small relative to $h$, the convergence rate can be of asymptotically $\mathcal{O}(h^2)$ order on uniform grids. Further, the RRM element uses piecewise quadratic polynomials (total degree not bigger than 2) and it is of the lowest degree ever possible for the model problem. 
~\\

The RRM element can be viewed as a reduction of the rectangular Morley (RM) element.  The application of the RM element to $H^{2}$ and $H^{1}$ problems has been studied in \cite{Wang.M;Shi.Z2013mono} and \cite{XY.Meng;XQ.Yang;S.Zhang2016}, respectively. In \cite{Wang;Xu;Hu2006}, it is proved that the RM element is uniformly convergent for the singular perturbation problem based on a modified variational functional. Later in \cite{Wang;Wu;Xie2013}, the uniform convergence rate of the original RM element for \eqref{eq:model problem} is shown. For the RRM element space, the consistency error estimate can be obtained consequently. The main difficulty then lies in estimating the approximation error. Similar to the elements described in \cite{Fortin.M;Soulie.M1983,Park.C;Sheen.D2003,Shuo.Zhang2020} and in many spline-type methods \cite{Schumaker2007,WangRenhong2013}, the number of continuity restrictions of the RRM element function associated with a cell is greater than the dimension of the local polynomial space. This element can not be constructed in the formulation of Ciarlet's triple and does not yield a natural nodal interpolation operator, though it does admit a set of locally supported basis functions. In \cite{Zeng.H;Zhang.C;Zhang.S2019}, the approximation of the RRM element space in the broken $H^{1}$ norm is analyzed with the help of a regularity result which is valid for convex domains. In the present paper, we reconstruct the estimation for the broken $H^{2}$ and $H^{1}$ norms on convex and non-convex domains. Inspired by the construction of quasi-spline interpolation operators in the spline function theory (see, for example, \cite{Wang;Lu1998,Sablonniere.P2003,2Sablonniere.P2003}), we propose a locally-averaging operator which preserves $P_{2}$ polynomials locally and is stable in terms of relevant Sobolev norms. Consequently, optimal error estimate of the interpolation operator is established. This interpolation operator is suitable for any regions that can be subdivided into rectangles, and particularly an optimal estimation can be given for the RRM element space in the broken $H^{1}$ norm on non-convex domains.  Therefore, the convergence analysis of the RRM element for the model problem~\eqref{eq:model problem} robustly in $\varepsilon$ is follows. 
~\\

It is notable that the newly-designed interpolation operator is not a projection onto the RRM element space. This is not surprising as the RRM element space does not correspond to a finite element defined as Ciarlet's triple, and it looks more like a nonconforming spline space. Actually, due to the importance of and the convenience introduced by the locally-defined projective operator, to the best of our knowledge, it has been a long-standing open problem to figure out an condition for the existence of interpolations which are stable, projective and locally defined in high (more than one) dimension \cite[Remark 1]{Demko1985}. To understand the situation, we establish a theory to give a necessary and sufficient condition when both projection and locality properties are satisfied simultaneously for an interpolator. An application of this theory to the RRM element space indicates that there exists no local interpolation with the given locally-supported basis functions which preserves the RRM element space. 
~\\

The rest of the paper is organized as follows. In Section~\ref{sec:pre}, some preliminaries are given and the rectangular Morley element is revisited. In Section~\ref{sec:rrmscheme}, the reduced rectangular Morley element space is revisited, some properties of the basis functions are presented, and the approximation analysis is conducted based on a locally-averaging interpolation operator constructed therein. In Section~\ref{sec:convergence}, the convergence analysis of the RRM element for the model problem~\eqref{eq:model problem} robustly in $\varepsilon$ is provided. In Section~\ref{sec:numerical}, some numerical results are presented to verify our theoretical findings. Finally, in the appendix, a necessary and sufficient condition for constructing an interpolation with a projective property is proposed. The theory indicates that there does not exist any local interpolation operator which preserves the RRM element space.

\section{Preliminaries} 
\label{sec:pre}

\subsection{Notations}  
We use $\nabla$ and $\nabla^{2}$ to denote the gradient operator and Hessian, respectively. We use standard notation on Lebesgue and Sobolev spaces, such as $L^{p}(\Omega)$, $H^{s}(\Omega)$, and $H^{s}_{0}(\Omega)$.  Denote, by $H^{-s}(\Omega)$, the dual spaces of $H^{s}_0(\Omega)$. We utilize the subscript $``\cdot_h"$ to indicate the dependence on grids. Particularly, an operator with the subscript $``\cdot_h"$ implies the operation is done cell by cell. Finally, $\lesssim$, $\gtrsim$, and $\cequiv$ respectively denote $\leqslant$, $\geqslant$, and $=$ up to a generic positive constant \cite{J.Xu1992}, which  might depend on the shape-regularity of subdivisions, but not on the mesh-size~$h$  and the perturbation parameter $\varepsilon$.

Let $\big\{\mathcal{G}_h\big\}$ be in a family of rectangular grids of domain $\Omega$. Let $\mathcal{N}_h$ be the set of all vertices, $\mathcal{N}_h=\mathcal{N}_h^i\cup\mathcal{N}_h^b$, with $\mathcal{N}_h^i$ and $\mathcal{N}_h^b$ comprising the interior vertices and the boundary vertices, respectively. Similarly, let $\mathcal{E}_h=\mathcal{E}_h^i\bigcup\mathcal{E}_h^b$ be the set of all the edges, with $\mathcal{E}_h^i$ and $\mathcal{E}_h^b$ comprising the interior edges and the boundary edges, respectively. If none of the vertices of a cell is on $\partial\Omega$, we name it an interior cell, otherwise it is called a boundary cell. We use $\mathcal{K}_{h}^{i}$ and $\mathcal{K}_{h}^{b}$ for the set of interior cells and boundary cells, respectively. Let $ \mathring{\omega}$ denote the interior of the region $\omega$. We use symbol $\#$ for the cardinal number of a set. For an edge $e$, $\mathbf{n}_e$ is a unit vector normal to $e$ and $\boldsymbol{\tau}_e$ is a unit tangential vector of $e$ such that $\mathbf{n}_e\times \boldsymbol{\tau}_e>0$. On the edge $e$, we use $\llbracket\cdot\rrbracket_e$ for the jump across $e$. If $e\subset\partial\Omega$, then $\llbracket\cdot\rrbracket_e$ is the evaluation on $e$. The subscript ${\cdot}_e$ can be dropped when there is no ambiguity brought in.

Suppose that $K$ represents a rectangle with sides parallel to the two axis respectively. Let $\{X_{i}\}\big|_{i =  1:4}$ and $\{e_{i}\}\big|_{i =  1:4}$ denote the sets of vertices and edges of $K$, respectively. Let $c_{K}$ be the barycenter of $K$. Let $L_{K}$, $H_{K}$ be the length of $K$ in the $x$ and $y$ directions, respectively.  Let $h_{K}: = \max\{L_{K},H_{K}\}$ be the size of $K$, and $\rho_{K}$ be the inscribed circle radius.
Let $h := \max\limits_{K \in \mathcal{G}_{h}}h_{K}$ be the mesh size of $\mathcal{G}_{h}$. Let $P_l(K)$ denote the space of all polynomials on $K$ with the total degree no more than~$l$. Let $Q_l(K)$ denote  the space of all polynomials on $K$ of degree no more than~$l$ in each variable. Similarly, we define spaces $P_l(e)$ and $Q_l(e)$ on an edge~$e$.

In this paper, we assume that $\big\{\mathcal{G}_{h}\big\}$ is in a regular family of rectangular grids of domain $\Omega$, i.e.,
\begin{equation}\label{eq:regularity}
\max_{K\in \mathcal{G}_{h}}\frac{h_{K}}{\rho_{K}} \leq \gamma_{0},
\end{equation}
where $\gamma_{0}$ is a generic constant independent of $h$.  Such a mesh is actually locally quasi-uniform, and this helps for the stability analysis of the interpolation operator constructed in Section~\ref{sec:rrmscheme}.

\subsection{Model problem and nonconforming finite element approximation}
The weak form of the model problem \eqref{eq:model problem} is given by : Find $u\in V : = H^{2}_{0}(\Omega)$ satisfying 
\begin{align}
\varepsilon^{2}a(u,v) + b(u,v) = (f,v), \quad \forall v \in H^{2}_{0}(\Omega),
\end{align}
where 
\begin{align*}
a(u,v) = \int_{\Omega} \nabla^{2} u :  \nabla^{2} v {\,d x d y } \quad \mbox{and}  \quad b(u,v) = \int_{\Omega} \nabla u \cdot  \nabla v {\,d x d y }.
\end{align*}
Given an discrete space $V_{h}$ defined on $\mathcal{G}_{h}$, the discrete weak formulation corresponding to~\eqref{eq:model problem} reads as: Find $u_{h}\in V_{h}$, such that
\begin{align}\label{eq:discrete form}
\varepsilon^{2}a_{h}(u_{h},v_{h}) + b_{h}(u_{h},v_{h}) = (f,v_{h}), \quad \forall v_{h} \in V_{h},
\end{align}
where 
\begin{align*}
a_{h}(u_{h},v_{h}) = \sum_{K\in \mathcal{G}_{h}}\int_{K} \nabla^{2} u_{h} :  \nabla^{2} v_{h} {\,d x d y }   \quad \mbox{and}  \quad b_{h}(u_{h},v_{h}) =\sum_{K\in \mathcal{G}_{h}}\int_{K} \nabla u_{h} \cdot  \nabla v_{h} {\,d x d y } .
\end{align*}

Let $u^{0}$ be the solution of the following boundary value problem:
\begin{equation}\label{eq:2nd problem}
\left\{
\begin{array}{rl}
  -\Delta u^{0} = f, & \mbox{ in } \ \Omega , \\
u^{0}   = 0, & \mbox{ on } \ \partial\Omega .
\end{array}
\right.
\end{equation}
The following regularity result is derived in \cite{Nilssen;Tai;Winther2001}.
\begin{lemma}{\rm(\!\cite[Lemma~5.1]{Nilssen;Tai;Winther2001})}\label{lem:regularity}
For a convex domain $\Omega$, there exist a constant $C$, independent of $\varepsilon$ and $f$, such that
\begin{align}
|u|_{2,\Omega} + \varepsilon |u|_{3,\Omega} & \leqslant C \varepsilon^{-\frac{1}{2}} \|f\|_{0,\Omega}; \\
 \big|u- u^{0}\big|_{1,\Omega} & \leqslant C \varepsilon^{\frac{1}{2}} \|f\|_{0,\Omega}.
\end{align}
\end{lemma}

\subsection{Rectangular Morley (RM) element}
The RM element is defined by 
 $(K,P_{K}^{\rm{M}},D_{K}^{\rm{M}})$ with the following properties:
\begin{itemize}
\item[(1)] $K$ is a rectangle;
\item[(2)] $P_{K}^{\rm{M}} = P_{2}(K) + \text{span}\{x^{3},y^{3}\}$;
\item[(3)] for any $v\in H^{2}(K)$, $D_{K}^{\rm{M}} =\big\{ v(X_{i}), \ \fint_{e_{i}}\partial_{\mathbf{n}_{e_{i}}}v \, d s \big\}_{i=1:4}$.
\end{itemize}
Given a grid $\mathcal{G}_h$, define the RM element space on $\mathcal{G}_h$ as
\begin{equation*}
\begin{split}
V_{h}^{\rm{M}}(\mathcal{G}_h) := \Big\{w_{h}\in L^{2}(\Omega) : w_{h}|_{K} \in P_{K}^{\rm{M}}, \ & w_{h}(X)\mbox{ is continuous at any nodes } X \in \mathcal{N}_{h}^{i}, \\
& \mbox{and} \fint_{e}\partial_{\mathbf{n}_{e}} w_{h} \, d s \mbox{ is continuous across any edge } e \in \mathcal{E}_{h}^{i} \Big\}.
\end{split}
\end{equation*}
Associated with the boundary condition of $H^{1}_{0}$ type, define $V_{hs}^{\rm{M}}(\mathcal{G}_h) :=\Big\{w_{h}\in V_{h}^{\rm{M}} :  w_{h}(X)=0, \ \forall X\in \mathcal{N}_{h}^{b}\Big\}$, and associated with the boundary condition of $H^{2}_{0}$ type, define $V_{h0}^{\rm{M}}(\mathcal{G}_h) :=\Big\{w_{h}\in V_{hs}^{\rm{M}} :\   \fint_{e}\partial_{\mathbf{n}_{e}} w_{h} \,d s =0,  \ \forall e\in \mathcal{E}_{h}^{b} \Big\}.$ In the sequel, we can drop the dependence on $\mathcal{G}_h$ when no ambiguity is brought in. 
\begin{lemma}{\rm(\!\cite[Theorem~5.4.1]{Wang.M;Shi.Z2013mono})}
It holds for any function $v_{h} \in V_{h0}^{\rm{M}}$ and $v\in H^{3}(\Omega)\cap H^{2}_{0}(\Omega)$ that
\begin{align*}
| a_{h}(v,v_{h}) - (\Delta^{2} v,v_{h}) | \lesssim h \big(|v|_{3,\Omega}+h\|\Delta^{2}v\|_{0,\Omega}\big) |v_{h}|_{2,h}.
\end{align*}
\end{lemma}
\begin{lemma}{\rm(\!\cite[Lemmas 3.2 and 3.5]{XY.Meng;XQ.Yang;S.Zhang2016})}\label{lem:consisRM} For any function $v_{h} \in V_{hs}^{\rm{M}}$, we have the following estimates:
\begin{itemize}
\item[(a)] For any shape-regular rectangular grid, it holds that
\begin{align*}
| b_{h}(v,v_{h}) + (\Delta v,v_{h}) | \lesssim \sum_{K\in \mathcal{G}_{h}} h_{K}^{2} |v|_{2,K} |v_{h}|_{2,K}\lesssim h |v|_{2,\Omega} |v_{h}|_{1,h}, \quad \forall v\in H^{2}(\Omega)\cap H^{1}_{0}(\Omega);
\end{align*}

\item[(b)] For any uniform rectangular grid, it holds that
\begin{align*}
| b_{h}(v,v_{h}) + (\Delta v,v_{h}) | \lesssim h^{k-1} |v|_{k,\Omega} |v_{h}|_{1,h}, \quad \forall v\in H^{k}(\Omega)\cap H^{1}_{0}(\Omega), \quad k = 2,3.
\end{align*}
\end{itemize}
\end{lemma}
\section{Reduced rectangular Morley element space revisited}
\label{sec:rrmscheme}

The reduced rectangular Morley (RRM) element space \cite{Shuo.Zhang2020, Zeng.H;Zhang.C;Zhang.S2019} is defined as  
\begin{equation*}
\begin{split}
V_{h}^{\rm{R}} := \Big\{w_{h}\in V^{\rm M}_h : w_{h}|_{K} \in P_{2}(K)\Big\}.
\end{split}
\end{equation*}
Associated with $H^{1}_{0}(\Omega)$, define $V_{hs}^{\rm{R}} :=V^{\rm R}_{h}\cap V^{\rm M}_{hs}$, and  associated with $H^{2}_{0}(\Omega)$, define $V_{h0}^{\rm{R}} :=V^{\rm M}_{h0}\cap V^{\rm R}_h$.

\subsection{Local basis functions of RRM element}
A boundary vertex is called a {\it corner node} if it is an intersection of two boundary edges that are not on the same line. It can be divided into {\it convex corner node} or {\it concave corner node}. It is  assumed that any two corner nodes are not in the same cell.   A boundary edge is called a {\it corner edge} if one of its endpoints is a corner node, otherwise it is named as {\it non-corner boundary edge}.  Denote, by $\mathcal{M}_{K}$,  a $3\times 3$ patch centered at $K$, whose lengths and heights are denoted as $\big\{L_{K,-1},\ L_{K},\ L_{K,1}\big\}$ and $\big\{H_{K,-1},\ H_{K},\ H_{K,1}\big\}$, respectively;  see Figure~\ref{fig:3x3basis}. 
\begin{figure}[!htbp]
\centering
\includegraphics[height=0.34\hsize]{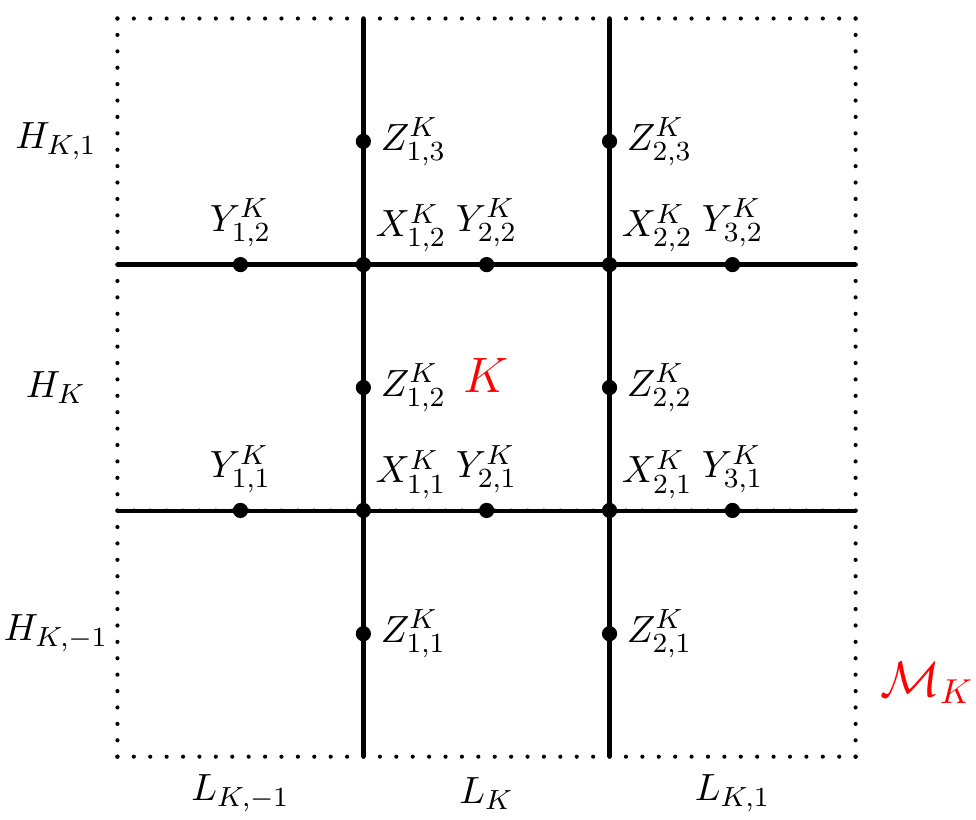}
\caption{Illustration of a $3 \times 3$ patch $\mathcal{M}_{K}$.}\label{fig:3x3basis}
\end{figure}
%
\begin{lemma}\rm(\!\cite[Lemma 15]{Shuo.Zhang2020})
Let $\mathcal{M}_{K}$ be a $3\times 3$ patch centered at $K$; see Figure~\ref{fig:3x3basis}.  Denote $V_{K}^{\rm{R}}:=V^{\rm R}_{h0}(\mathcal{M}_{K})$. Then ${\rm dim}(V_{K}^{\rm{R}} ) = 1$.
\end{lemma}
Now we give a detailed description of the functions in $V_{K}^{\rm{R}}$. Let $\big\{X_{m,n}^{K}\big\}$, $\big\{Y_{m,n}^{K}\big\}$, and $\big\{Z_{m,n}^{K}\big\}$ denote the  interior vertices,  interior edge midpoints in the $x$ direction, and interior edge midpoints in the $y$ direction inside $\mathcal{M}_{K}$, respectively (see Figure~\ref{fig:3x3basis}). For any $\varphi \in V_{K}^{R}$, we denote $v_{m,n}^{K}:=\varphi(X_{m,n}^{K})$, $u_{m,n}^{K}:=\partial_y\varphi(Y_{m,n}^{K})$, and $z_{m,n}^{K}:=\partial_x\varphi(Z_{m,n}^{K})$. Then the values of $\{v_{m,n}^{K}\}$,  $\{u_{m,n}^{K}\}$, and  $\{z_{m,n}^{K}\}$ satisfy that
\begin{align}
\big[v_{1,1}^{K},\ v_{2,1}^{K},\ v_{1,2}^{K},\ v_{2,2}^{K}\big] & = \big[1,  \gamma_{x}^{K},\ \gamma_{y}^{K},\ \gamma_{x}^{K}\gamma_{y}^{K}\big]v_{1,1}^{K}; \label{eq:values1Basis} 
\\
\big[u_{1,1}^{K},\ u_{2,1}^{K},\ u_{3,1}^{K},\ u_{1,2}^{K},\ u_{2,2}^{K},\ u_{3,2}^{K}\big] & = \Big[\tfrac{1}{H_{K,-1}},\ \tfrac{1+\gamma_{x}^{K}}{H_{K,-1}},\ \tfrac{\gamma_{x}^{K}}{H_{K,-1}},\ \tfrac{-\gamma_{y}^{K}}{H_{K,1}},\ \tfrac{-(1+\gamma_{x}^{K})\gamma_{y}^{K}}{H_{K,1}},\ \tfrac{-\gamma_{x}^{K}\gamma_{y}^{K}}{H_{K,1}}\Big]v_{1,1}^{K};\label{eq:values2Basis}
\\
\big[z_{1,1}^{K},\ z_{2,1}^{K},\ z_{1,2}^{K},\ z_{2,2}^{K},\ z_{1,3}^{K},\ z_{2,3}^{K}\big] &= \Big[\tfrac{1}{L_{K,-1}},\ \tfrac{-\gamma_{x}^{K}}{L_{K,1}},\ \tfrac{1+\gamma_{y}^{K}}{L_{K,-1}},\ \tfrac{-(1+\gamma_{y}^{K})\gamma_{x}^{K}}{L_{K,1}},\ \tfrac{\gamma_{y}^{K}}{L_{K,-1}},\ \tfrac{-\gamma_{x}^{K}\gamma_{y}^{K}}{L_{K,1}}\Big]v_{1,1}^{K},\label{eq:values3Basis}
\end{align}
where $\gamma_{x}^{K}= \frac{1+\frac{L_{K}}{L_{K,-1}}}{1+\frac{L_{K}}{L_{K,1}}}$ and $\gamma_{y}^{K}= \frac{1+\frac{H_{K}}{H_{K,-1}}}{1+\frac{H_{K}}{H_{K,1}}}$. For each vertice $X_{m,n}^{K}$, midpoint $Y_{m,n}^{K}$, or midpoint $Z_{m,n}^{K}$ on the boundary of $\mathcal{M}_{K}$, $v_{m,n}^{K}$, $u_{m,n}^{K}$, or $z_{m,n}^{K}$ equals to zero correspondingly. Therefore, $\varphi \in V_{K}^{\rm{R}}$ is uniquely determined, if $\varphi(X_{1,1}^{K})$ is determined. 
\begin{definition}\label{def:3x3basis}{
\rm
Let $\mathcal{M}_{K}$ be a $3\times 3$ patch with a center element $K$.  Denote, by $\varphi_{K}$, a function supported on $\mathcal{M}_{K}$, which satisfies
\begin{itemize}
\item[(a)]$\varphi_{K}(x,y)\equiv 0,\quad \forall (x,y)\notin  \mathcal{M}_{K}$;
\item[(b)] $\varphi_{K}|_{\mathcal{M}_{K}} \in V_{K}^{\rm{R}}$, and specially, $\varphi_{K}(X_{1,1}^{K}) = \frac{L_{K,-1}}{L_{K,-1}+L_{K}}\cdot \frac{H_{K,-1}}{H_{K,-1}+H_{K}}$.
	\end{itemize}	
}
\end{definition}

The assumption of $\varphi_{K}(X_{1,1}^{K}) = \frac{L_{K,-1}}{L_{K,-1}+L_{K}}\cdot \frac{H_{K,-1}}{H_{K,-1}+H_{K}}$ is not necessary, but can facilitate the subsequent analysis. 
\begin{figure}[!htbp]
\centering
\includegraphics[height=0.34\hsize]{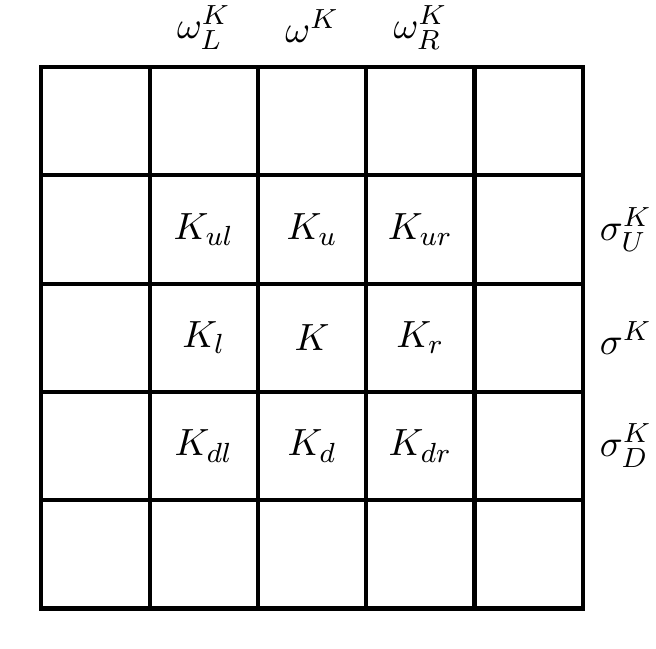}
\caption{Illustration of a $5 \times 5$ patch centered at $K$.}\label{fig:5x5basis}
\end{figure}

Recall that we use $\mathcal{K}_{h}^{i}$ and $\mathcal{K}_{h}^{b}$ for the set of interior cells and boundary cells, respectively. For any $K\in \mathcal{K}_{h}^{i}$, there exists a $3\times 3$ patch centered at $K$.

\begin{proposition}\label{pro:property of phi 5x5}
Let $K$ be an interior element in  $\mathcal{K}_{h}^{i}$, and $\mathcal{M}_{K}$ be its corresponding $3\times 3$ patch. Assume that the $5\times 5$ patch centered at $K$ is within $\Omega$. Let $A_{K}= \{K_{dl},\ K_{d},\ K_{dr},\ K_{l},\ K,\ K_{r},\ K_{ul},\ K_{u}, \ K_{ur} \};$ see Figure~\ref{fig:5x5basis}. Then $K$ is located in the supports of nine functions $\big\{\varphi_{T}:  \ T\in A_{K}\big\}$.
\begin{itemize}
\item[(a)] For any $v\in Q_{1}(\mathcal{M}_{K})$ and $(x,y)\in K$, it holds that
\begin{align*}
\sum_{T\in A_{K}} v(c_{T})\varphi_{T}(x,y) = v \quad \mbox{and}\quad 
\sum_{T\in A_{K}} (\fint_{T}v{\,d x d y }) \varphi_{T}(x,y)= v.
\end{align*}
\item[(b)] For any $v \in P_{2}(\mathcal{M}_{K})$ and $(x,y)\in K$, it holds that
\begin{align*}
&\sum_{T\in A_{K}} r_{T}(v)\varphi_{T}(x,y) = v, \mbox{ with} \ \  r_{T}(v) = v(c_{T}) - \tfrac{1}{8}\big(L_{T}^{2}  \tfrac{\partial^{2}v}{\partial x^{2}} + H_{T}^{2}\tfrac{\partial^{2}v}{\partial y^{2}} \big),
\\
&\sum_{T\in A_{K}} t_{T}(v)\varphi_{T}(x,y) =v, \mbox{ with} \ \ t_{T}(v) = \fint_{T}v {\,d x d y }  - \tfrac{1}{6}\big(L_{T}^{2}\tfrac{\partial^{2}v}{\partial x^{2}}  + H_{T}^{2}\tfrac{\partial^{2}v}{\partial y^{2}} \big).
\end{align*}
\item[(c)] For any $(x,y)\in K$, there exists a set of coefficients $\{d_{T}\}$, such that 
\begin{align*}
\sum_{T\in A_{K}} d_{T}L_{T}H_{T}\varphi_{T}(x,y) = 0, \quad \forall (x,y)\in K,
\end{align*}
where $d_{K} = \pm 1$, $d_{K_{d}} = d_{K_{u}} = d_{K_{l}} = d_{K_{r}} = - d_{K}$, and $d_{K_{dl}} = d_{K_{dr}} = d_{K_{ul}} = d_{K_{ur}}  = d_{K}$.
\end{itemize}
\end{proposition}
\begin{proof} 
$(a)$
The left-hand-side of each equality is a sum of $P_{2}$ polynomials restricted on $K$, and the right-hand-side is a bilinear polynomial. We only need to verify that their values equal on the vertices of $K$, and normal derivatives equal on the midpoints of edges on $\partial K$. Utilizing \eqref{eq:values1Basis}, \eqref{eq:values2Basis}, \eqref{eq:values3Basis}, and $\varphi_{T}(X_{1,1}^{T}) = \frac{L_{T,-1}}{L_{T,-1}+L_{T}}\cdot \frac{H_{T,-1}}{H_{T,-1}+H_{T}}$, the results can be verified directly.

\noindent$(b)$
For $v = x^{2}$, direct calculation leads to, \begin{align*}
& \sum_{T\in A_{K}}  v(c_{T})\varphi_{T}(x,y) = x^{2} + \frac{1}{4}\sum_{T\in A_{K}} L_{T}^{2}\varphi_{T}(x,y), \quad \forall (x,y)\in K; \\
& \sum_{T\in A_{K}} (\fint_{T}v{\,d x d y }) \varphi_{T}(x,y)  = x^{2} + \frac{1}{3}\sum_{T\in A_{K}} L_{T}^{2}\varphi_{T}(x,y), \quad \forall (x,y)\in K.
\end{align*}
For $v = y^{2}$, we have similarly,
\begin{align*}
& \sum_{T\in A_{K}} v(c_{T})\varphi_{T}(x,y) = y^{2} + \frac{1}{4}\sum_{T\in A_{K}} H_{T}^{2}\varphi_{T}(x,y), \quad \forall (x,y)\in K; \\
& \sum_{T\in A_{K}} (\fint_{T}v{\,d x d y }) \varphi_{T}(x,y)  = y^{2} + \frac{1}{3}\sum_{T\in A_{K}} H_{T}^{2}\varphi_{T}(x,y), \quad \forall (x,y)\in K.
\end{align*}
Therefore, for $v \in P_{2}(\mathcal{M}_{K})$, it holds that, 
\begin{align*}
&  \sum_{T\in A_{K}} r_{T}(v)\varphi_{T}(x,y) = v, \ \forall (x,y)\in K, \mbox{ where } r_{T}(v) = v(c_{T}) - \tfrac{1}{8}\big(L_{T}^{2}  \tfrac{\partial^{2}v}{\partial x^{2}} + H_{T}^{2}\tfrac{\partial^{2}v}{\partial y^{2}} \big);
\\
& \sum_{T\in A_{K}} t_{T}(v)\varphi_{T}(x,y) = v,  \ \forall (x,y)\in K, \mbox{ where } t_{T}(v) = \fint_{T}v{\,d x d y} - \tfrac{1}{6}\big(L_{T}^{2}\tfrac{\partial^{2}v}{\partial x^{2}}  + H_{T}^{2}\tfrac{\partial^{2}v}{\partial y^{2}} \big).
\end{align*}

\noindent$(c)$
Similar to $(a)$, consider the left-hand-side of this equality, we only have to verify its values vanish at each vertex of $K$, and its normal derivatives vanish at each midpoints of edges of $K$. 
\end{proof}

\begin{remark}\label{rem:adjoint basis property}
{\rm
Suppose $K$ is located in the supports of nine functions $\big\{\varphi_{T}:  \ T\in A_{K}\big\}$;  see Figure~\ref{fig:5x5basis}. Denote $\omega_{L}^{K}:= K_{dl}\cup K_{l}\cup K_{ul}$, $\omega^{K}:= K_{d}\cup K\cup K_{u}$, and $\omega_{R}^{K}:= K_{dr}\cup K_{r}\cup K_{ur}$. Denote $\sigma_{D}^{K} := K_{dl}\cup K_{d}\cup K_{dr}$, $\sigma^{K} := K_{l}\cup K \cup K_{r}$, and $\sigma_{U}^{K} := K_{ul}\cup K_{u}\cup K_{ur}$. For $v \in \{1,\ x, \ y,\ xy, \ x^{2}, \ y^{2}\}$, it can be shown that 
}
\begin{itemize}
\item[(a)]
{\rm
$\big(r_{K}(v)\varphi_{K} + r_{K_{r}}(v)\varphi_{K_{r}}\big)|_{\omega_{L}^{K}\cup \omega^{K}}$,  $\big(t_{K}(v)\varphi_{K} + t_{K_{r}}(v)\varphi_{K_{r}}\big)|_{\omega_{L}^{K}\cup \omega^{K}}$, and $\big(d_{K}L_{K}H_{K}\varphi_{K} + d_{K_{r}} L_{K_{r}}H_{K_{r}}\varphi_{K_{r}}\big)|_{\omega_{L}^{K}\cup \omega^{K}}$ are independent of $L_{K,1} \ \big( =L_{K_{r}}\big)$ and $L_{K_{r},1}.$ 
}
\item[(b)]
{\rm
$\big(r_{K_{l}}(v)\varphi_{K_{l}} + r_{K}(v)\varphi_{K}\big)|_{\omega^{K} \cup \omega_{R}^{K}}$, $\big(t_{K_{l}}(v)\varphi_{K_{l}} + t_{K}(v)\varphi_{K}\big)|_{\omega^{K}\cup \omega_{R}^{K}}$ and $\big(d_{K_{l}}L_{K_{l}}H_{K_{l}}\varphi_{K_{l}} + d_{K}L_{K}H_{K}\varphi_{K}\big)_{\omega^{K}\cup \omega_{R}^{K}}$ are  independent of $L_{K_{l},-1}$ and $L_{K,-1} \ \big(= L_{K_{l}}\big).$
}
\item[(c)]
{\rm 
$\big(r_{K}(v)\varphi_{K} + r_{K_{u}}(v)\varphi_{K_{u}}\big)|_{\sigma_{D}^{K}\cup \sigma^{K}}$, $\big(t_{K}(v)\varphi_{K} + t_{K_{u}}(v)\varphi_{K_{u}}\big)|_{\sigma_{D}^{K}\cup \sigma^{K}}$ and $\big(d_{K}L_{K}H_{K}\varphi_{K} +d_{K_{u}}L_{K_{u}}H_{K_{u}}\varphi_{K_{u}}\big)_{\sigma_{D}^{K}\cup \sigma^{K}}$ are independent of $H_{K,1} \ \big(= H_{K_{u}}\big)$ and $H_{K_{u},1}.$
}
\item[(d)]
{\rm
$\big(r_{K_{d}}(v)\varphi_{K_{d}} + r_{K}(v)\varphi_{K}\big)|_{\sigma^{K}\cup \sigma_{U}^{K}}$, $\big(t_{K_{d}}(v)\varphi_{K_{d}} + t_{K}(v)\varphi_{K}\big)|_{\sigma^{K}\cup \sigma_{U}^{K}}$ and $\big(d_{K_{d}}L_{K_{d}}H_{K_{d}}\varphi_{K_{d}} + d_{K}L_{K}H_{K}\varphi_{K}\big)_{\sigma^{K}\cup \sigma_{U}^{K}}$ are independent of $H_{K_{d},-1}$ and $H_{K,-1}  \ \big(=H_{K_{d}}\big).$
}
\end{itemize}
\end{remark}
\begin{figure}[!htbp]
\centering
\includegraphics[height=0.33\hsize]{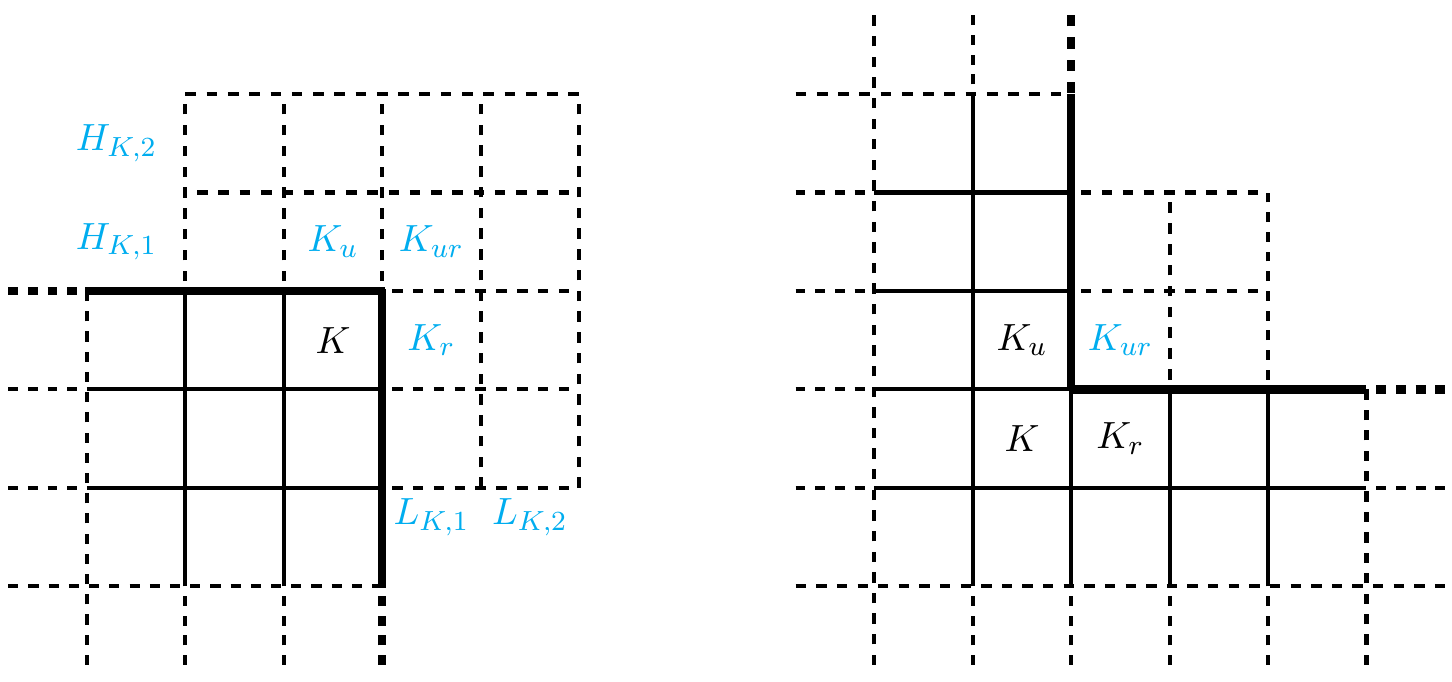}
\caption{Expansion outside $\Omega$ and four basis functions added at a convex corner node (Left) and a concave corner node (Right).}\label{fig:expansion at a corner}
\end{figure}
\begin{figure}[!htbp]
\centering
\includegraphics[height=0.33\hsize]{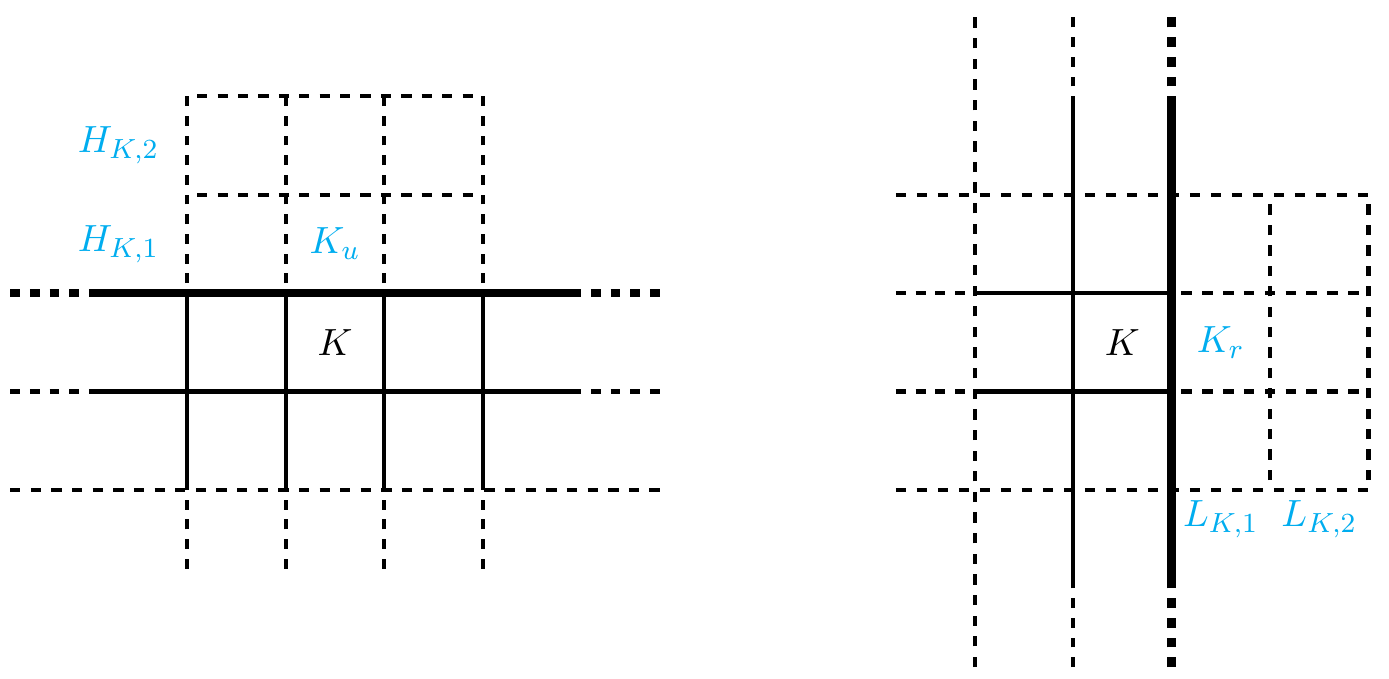}
\caption{Expansion outside a non-corner boundary edge and two basis functions added.}\label{fig:expansion at a non-corner edge}
\end{figure}
$\big\{ \mathcal{M}_{K} \big\}_{K\in \mathcal{K}_{h}^{i}}$ represents a set of patches within $\Omega$, and $\big\{ \varphi_{K} \big\}_{K\in \mathcal{K}_{h}^{i}}$ is a set of functions in $V_{h0}^{\rm R}$. Traversing all the corner nodes and non-corner boundary edges, we get a set of newly added functions $\{\varphi_{K}\}$ by the expansions below. These functions, if restricted on $\Omega$, are in the space $V_{h}^{\rm R}\backslash V_{h0}^{\rm R}$.
\begin{itemize}
\item Consider a convex corner as shown in Figure~\ref{fig:expansion at a corner} (Left). Let $L_{K,1}$, $L_{K,2}$, $H_{K,1}$, and $H_{K,2}$ be some constants close to   $h_{K}$.  Complete a $3\times 3$ patch, denoted by $\mathcal{M}_{K}$, outside the domain with $K$ as the center. The element to the right of $K$ is denoted as $K_{r}$, the element above it is denoted as $K_{u}$, and the element opposite to $K$ with respect to the corner node is denoted as $K_{ur}$. Adding a layer of rectangles outside $\mathcal{M}_{K}$, we obtain four patches $\mathcal{M}_{K}$, $\mathcal{M}_{K_{r}}$, $\mathcal{M}_{K_{u}}$, and $\mathcal{M}_{K_{ur}}$ associated with this convex corner. And four functions supported on them are denoted as $\varphi_{K}$, $\varphi_{K_{r}}$, $\varphi_{K_{u}}$, and $\varphi_{K_{ur}}$, respectively. 

\item Consider a concave corner as shown in Figure~\ref{fig:expansion at a corner} (Right). We also extend the mesh to get four patches, each of which is centered at $K$, $K_{r}$, $K_{u}$, and an added element $K_{ur}$, and we derive four functions supported on four $3\times 3$ patches correspondingly. 

\item Consider a non-corner boundary edge shown in Figure~\ref{fig:expansion at a non-corner edge} (Left). Let $H_{K,1}$ and $H_{K,2}$ be two arbitrary constants close to the height of $K$. A $3\times 3$ patch $\mathcal{M}_{K}$, is completed outside the domain centered at $K$. The element opposite to $K$ with respect to the non-corner boundary edge is denoted as $K_{u}$. Extending a layer of rectangles outside $\mathcal{M}_{K}$, a $3\times 3$ patch centered at $K_{u}$ is derived and named as $\mathcal{M}_{K_{u}}$.  
Let $\varphi_{K}$ and $\varphi_{K_{u}}$ denote two functions supported on $\mathcal{M}_{K}$ and $\mathcal{M}_{K_{u}}$, respectively. Similar operations are conducted on the non-corner boundary edge in the vertical direction; see Figure~\ref{fig:expansion at a non-corner edge} (Right). 
\end{itemize}

The above expanding operations are carried out locally, by which each element in $\mathcal{G}_{h}$ can be located in the supports of nine functions. For each boundary cell $K$, the choice of $L_{K,1}$, $L_{K,2}$, $H_{K,1}$, and $H_{K,2}$ appeared in Figures~\ref{fig:expansion at a corner} and \ref{fig:expansion at a non-corner edge}
can be freely determined only according to the size of $K$, such that \eqref{eq:regularity} still holds. Let $\mathcal{K}_{h}^{\rm ex}$ be the set of all newly added elements near corner nodes and non-corner boundary edges, such as $K_{u}$, $K_{r}$, $K_{ur}$ in Figures~\ref{fig:expansion at a corner} and \ref{fig:expansion at a non-corner edge}. Denote $\mathcal{B}_{h} = \mathcal{K}_{h}^{b} \cup \mathcal{K}_{h}^{\rm ex}$. Then $\big\{ \mathcal{M}_{K} \big\}_{K \in \mathcal{B}_{h}}$ consists of patches not completely contained in $\Omega$. Define $\mathcal{J}_{h}: = \mathcal{K}_{h}^{i} \cup \mathcal{B}_{h}$. In the spirit of Theorem 17 in \cite{Shuo.Zhang2020}, we have
\begin{align}
V_{h}^{\rm{R}}  & = {\rm span}\big\{\varphi_{K}|_{\Omega}: K \in \mathcal{J}_{h}\big\};\\
V_{h0}^{\rm{R}}  & = {\rm span}\big\{\varphi_{K}: K \in \mathcal{K}_{h}^{i}\big\}.
\end{align}
Here $\big\{\varphi_{K} \big\}_{K\in \mathcal{K}_{h}^{i}}$ is a set of linearly independent basis functions in $V_{h0}^{\rm R}$,  and ${\rm dim}(V_{h0}^{\rm R}) = \#(\mathcal{K}_{h}^{i})$. Whereas, $\big\{\varphi_{K} \big\}_{K\in \mathcal{J}_{h}}$ is linearly dependent when these functions are restricted in $\Omega$; see Lemma~\ref{lem:property of phi}~$(c)$.
%
\begin{figure}[!htbp]
\centering
\includegraphics[height=0.32\hsize]{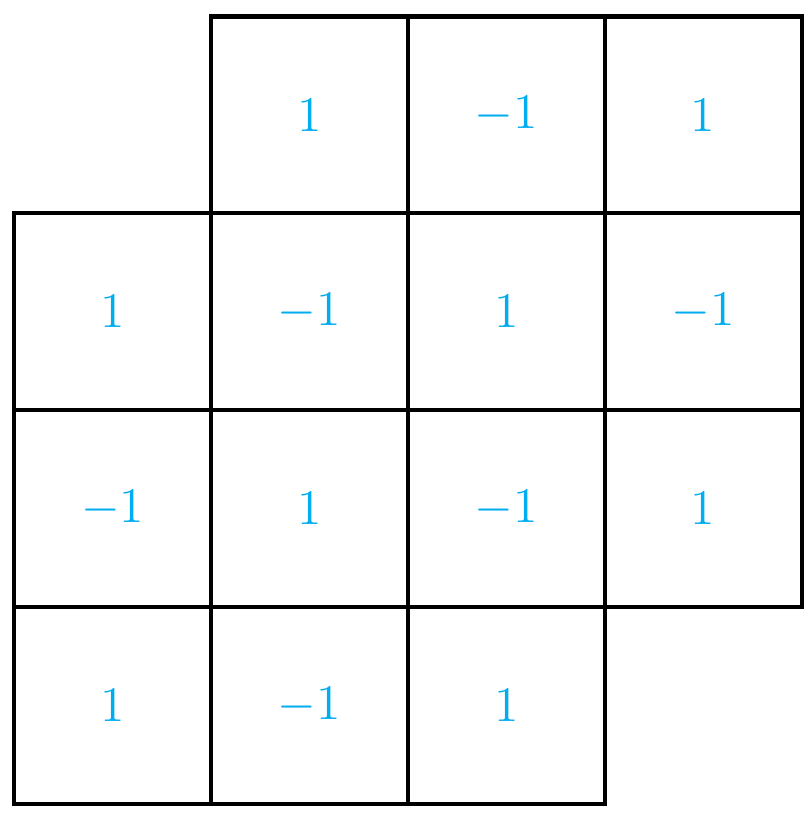}
\caption{Illustration of checkerboard distribution patterns.}\label{fig:checkerboard}
\end{figure}
%
\begin{lemma}\label{lem:property of phi}
Let $\widetilde{\Omega}_{h} = \cup_{K\in \mathcal{J}_{h}}\mathcal{M}_{K}$. The set of functions $\big\{\varphi_{K}\big\}_{K \in \mathcal{J}_{h}}$ have the following properties:
\begin{itemize}
\item[(a)] for any $v\in Q_{1}(\widetilde{\Omega}_{h})$, it holds that
$$
\sum_{K\in \mathcal{J}_{h}}v(c_{K})\varphi_{K}(x,y) = v \quad \mbox{and} \quad \sum_{K\in \mathcal{J}_{h}}(\fint_{K} v{\,d x d y }) \varphi_{K} = v, \quad \forall (x,y)\in \Omega;
$$
\item[(b)] for any $v \in P_{2}(\widetilde{\Omega}_{h})$, it holds that
\begin{align*}
&\sum\limits_{K\in \mathcal{J}_{h}}r_{K}(v)\varphi_{K}(x,y) = v, \ \mbox{with} \  r_{K}(v) = v(c_{K}) - \tfrac{1}{8}\big(L_{K}^{2}  \tfrac{\partial^{2}v}{\partial x^{2}} + H_{K}^{2}\tfrac{\partial^{2}v}{\partial y^{2}} \big), \quad \forall (x,y)\in \Omega,
\\
& \sum\limits_{K\in \mathcal{J}_{h}}t_{K}(v)\varphi_{K}(x,y)=v, \ \mbox{with} \   t_{K}(v) = \fint_{K}v{\,d x d y}  - \tfrac{1}{6}\big(L_{K}^{2}\tfrac{\partial^{2}v}{\partial x^{2}}  + H_{K}^{2}\tfrac{\partial^{2}v}{\partial y^{2}} \big), \quad \forall (x,y)\in \Omega;
\end{align*}
\item[(c)] 
with a set of coefficients $\mathcal{C}_{\mathcal{J}_{h}} = \{d_{K}\}_{K\in\mathcal{J}_{h}}$, named as a checkerboard coefficients set, which satisfies: 
(i) $d_{K} = \pm 1$, (ii) $d_{K_{l}} = d_{K_{r}} = d_{K_{d}} = d_{K_{u}} = d_{K}, \ \forall d_{K} \in \mathcal{C}_{\mathcal{J}_{h}}$ (see Figure~\ref{fig:checkerboard}),
it holds that 
\begin{align*}
\sum\limits_{K\in \mathcal{J}_{h}} d_{K} L_{K}H_{K}\varphi_{K}(x,y) = 0, \quad \forall (x,y)\in \Omega.
\end{align*}
\end{itemize}
\end{lemma}
\begin{proof}
It is equivalent to prove these results on each element in $\mathcal{G}_{h}$. By Proposition~\ref{pro:property of phi 5x5}, we only have to verify these equalities for the outermost two layers of elements in $\mathcal{G}_{h}$. Notice that the expanding operations are carried out locally, and each element in $\mathcal{G}_{h}$ is located in the supports of nine functions  $\big\{\varphi_{K}\big\}_{K\in \mathcal{J}_{h}}$. Take a right boundary as an example; see Figure~\ref{fig:expansion at a non-corner edge} (Right). According to Remark~\ref{rem:adjoint basis property}, the choices of $L_{K,1}$ and $L_{K,2}$ do not affect the values of $\big(r_{K}(v)\varphi_{K} + r_{K_{r}}(v)\varphi_{K_{r}}\big)|_{\Omega}$,  $\big(t_{K}(v)\varphi_{K} + t_{K_{r}}(v)\varphi_{K_{r}}\big)|_{\Omega}$, and $\big(d_{K}L_{K}H_{K}\varphi_{K} + d_{K_{r}}L_{K_{r}}H_{K_{r}}\varphi_{K_{r}}\big)|_{\Omega}$. Therefore, although these boundary elements on the same column may be extended outside $\Omega$ with different lengths, properties (a)-(c) stated in Property~\ref{pro:property of phi 5x5} is also true for elements 
located in the right outermost two layers of $\mathcal{G}_{h}$. The case of other boundaries can be verified similarly. 
\end{proof}
\begin{proposition}
Let $\varphi_{K}$ be the function supported on $\mathcal{M}_{K}$. It holds that
\begin{equation}\label{eq:normEstimate}
\big|\varphi_{K}|_{k,T}\leq C_{\gamma_{0}}h_{T}^{1-k}, \quad \forall T \subset \mathcal{M}_{K}, 
\end{equation}
where $C_{\gamma_{0}}$ represents a positive constant only dependent on the regularity constant $\gamma_{0}$.
\end{proposition}
\begin{proof}
For $T\subset \mathcal{M}_{K}$, $\varphi_{K}|_{T}\in P_{2}(T)$ and it can be written as
\begin{equation}\label{eq:RRMbasisOnK}
\varphi_{K}|_{T} = \sum_{s=1:4}a_{s}p_{s}^{\rm{M}} + \sum_{t=5:8}b_{t}q_{t}^{\rm{M}},
\end{equation}
where $p_{s}^{\rm{M}}$ and $q_{t}^{\rm{M}}$ represent the rectangular Morley basis functions related to nodes and edges of $T$, respectively. 
It is known that
\begin{equation}\label{eq:normRMbasis}
\big|p_{s}^{\rm{M}}\big|_{k,T} \leq C_{\gamma_{0}}h_{T}^{1-k} \quad \mbox{and} \quad  \big|q_{t}^{\rm{M}}\big|_{k,T} \leq C_{\gamma_{0}}h_{T} h_{T}^{1-k}.
\end{equation}
From \eqref{eq:values1Basis}--\eqref{eq:values3Basis}, we have
\begin{equation}\label{eq:valuesOnK}
a_{s} \leq C_{\gamma_{0}}, \quad b_{t} \leq C_{\gamma_{0}}h_{T}^{-1}.
\end{equation}
A combination of \eqref{eq:RRMbasisOnK}, \eqref{eq:normRMbasis}, and \eqref{eq:valuesOnK} leads to the desired result.
\end{proof}
\subsection{Interpolation operator for RRM element space}
We establish an available interpolation operator that is stable and reproduces quadratic polynomial.  Its construction is similar to the quasi-interpolation operators proposed in the spline theory \cite{Wang;Lu1998,Sablonniere.P2003,2Sablonniere.P2003}. As a matter of fact, an interpolation which does not necessarily preserve the entire finite element space and preserving quadratic polynomials is enough for the approximation property.

\begin{definition}\label{def:interpolationVh0}
With the set of functions $\{\varphi_{K}\}_{K\in \mathcal{K}_{h}^{i}}$, we define an interpolation operator for the homogeneous space
\begin{align*}
\Pi_{h0}: L^{1}(\Omega) \rightarrow V_{h0}^{\rm{R}},\quad \Pi_{h0}v = \sum\limits_{K\in \mathcal{K}_{h}^{i}}\lambda_{K}(v)\varphi_{K}(x,y),
\end{align*}
with $\lambda_{K}(v) = \sum\limits_{\mu = 1}^{5}w_{\mu}^{K}(\fint_{S_{\mu}^{K}}v{\,d x d y })$, 
where $\big\{S_{\mu}^{K}\}_{\mu = 1:5}$ are five cells around $K$ (see Figure~\ref{fig:S3x3basis}) and 
\begin{align*}
&w_{1}^{K} = \frac{-L_{K}^{2}}{(L_{K,-1}+L_{K})(L_{K,-1}+L_{K}+L_{K,1})}, \quad w_{2}^{K} = \frac{-L_{K}^{2}}{(L_{K}+L_{K,1})(L_{K,-1}+L_{K}+L_{K,1})},\\
&w_{3}^{K} = \frac{-H_{K}^{2}}{(H_{K,-1}+H_{K})(H_{K,-1}+H_{K}+H_{K,1})}, \quad w_{4}^{K} = \frac{-H_{K}^{2}}{(H_{K}+H_{K,1})(H_{K,-1}+H_{K}+H_{K,1})},\\
& w_{5}^{K} = 1 - (w_{1}^{K} + w_{2}^{K}  + w_{3}^{K}  + w_{4}^{K}).
\end{align*}
\begin{figure}[!htbp]
\centering
\includegraphics[height=0.34\hsize]{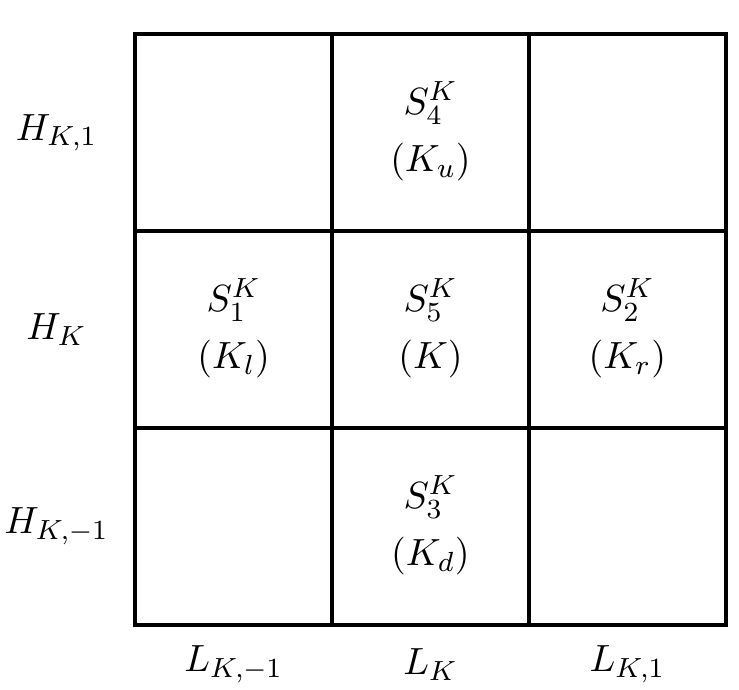}
\caption{Integral mean values on five elements around $K$.}\label{fig:S3x3basis}
\end{figure}
\end{definition}
The main result of this subsection is the theorem below.
\begin{theorem}\label{thm:approxH02}
Let $v\in H^{2}_{0}(\Omega) \cap H^{s}(\Omega)$, $2\leq s \leq 3$. Then $|v-\Pi_{h0}v|_{k,h}\lesssim h^{s-k}|v|_{s,\Omega}$ with $ 0\leq k \leq s.$ Morever, if $v\in H^{1}_{0}(\Omega)$, it also holds that $|v-\Pi_{h0}v|_{k,h}\lesssim h^{1-k}|v|_{1,\Omega}$, with $ k = 0,1.$
\end{theorem}

We postpone the proof of Theorem \ref{thm:approxH02} after several technical lemmas. We firstly introduce an auxiliary operator $\widetilde{\Pi}_{h}$.

\begin{definition}\label{def:interpolationVh}
Let $\widetilde{\Omega}_{h} = \cup_{K\in \mathcal{J}_{h}}\mathcal{M}_{K}$. With the set of functions $\big\{\varphi_{K}\big\}_{K\in \mathcal{J}_{h}}$, we define an interpolation operator 
\begin{align*}
&\widetilde{\Pi}_{h}: L^{1}(\widetilde{\Omega}_{h}) \rightarrow V_{h}^{\rm{R}}, \quad \widetilde{\Pi}_{h}(v) = \sum\limits_{K\in \mathcal{J}_{h}}\lambda_{K}(v)\varphi_{K}(x,y).
\end{align*}
\end{definition}

Since every functional $\lambda_{K}(v)$ in Definitions~ \ref{def:interpolationVh} and \ref{def:interpolationVh0} only involves the information of $v$ within $\mathcal{M}_{K}$, operators $\widetilde{\Pi}_{h}$ and $\Pi_{h0}$ define local approximation schemes. The interpolation $\widetilde{\Pi}_{h}(v)$ involves information of $v$ outside $\Omega$, and the difference between $(\widetilde{\Pi}_{h}v)|_{\Omega}$ and $(\Pi_{h0}v)|_{\Omega}$ only lies in some elements near $\partial\Omega$. 
\begin{lemma}\label{lem:preserveP2}%
Interpolations $\widetilde{\Pi}_{h}$ and $\Pi_{h0}$ preserve quadratic functions, namely,
\begin{itemize}
\item[(a)]
for any $T\in \mathcal{G}_{h}$, $(\widetilde{\Pi}_{h}v)|_{T} = v|_{T}$ with $v\in P_{2}(\widetilde{\Omega}_{h})$;
\item[(b)]
if $T\in \mathcal{G}_{h}$ and it satisfies $\#\big\{\mathcal{M}_{K}:\ \mathcal{M}_{K}\cap \mathring{T} \ne \varnothing, K\in \mathcal{K}_{h}^{i}\big\} = 9$, then $(\Pi_{h0}v)|_{T} = v|_{T}$ with $v\in P_{2}(\Omega)$.
\end{itemize}
\end{lemma}
\begin{proof}
\noindent $(a)$ By Lemma~\ref{lem:property of phi} and the difference theory, we replace the second derivatives appearing in the expression of $t_{K}(v)$ with a weighted sum of five integral mean values around $K$, where the weights are computed to be $\big\{w_{\mu}^{K}\big\}_{\mu = 1:5}$. That is to say,  for $v\in P_{2}(\widetilde{\Omega}_{h})$, it holds that $\lambda_{K}(v) = t_{K}(v)$. Therefore, $(\widetilde{\Pi}_{h} v)|_{T} = v|_{T}$ for any $T \in \mathcal{G}_{h}$.

\noindent $(b)$
The condition of $\#\big\{\mathcal{M}_{K}:\ \mathcal{M}_{K}\cap \mathring{T} \ne \varnothing, K\in \mathcal{K}_{h}^{i}\big\} = 9$ is to ensure $V_{h0}^{\rm R}\big|_{T} = V_{h}^{\rm R}\big|_{T}$, and the result is direct obtained from the proof in (a).
\end{proof}
\begin{lemma}\label{lem:stability}
Let $T\in \mathcal{G}_{h}$. Denote $\Delta_{T}  := \cup_{ \substack{ \mathcal{M}_{K}\supset T \\ K \in \mathcal{J}_{h}} } \mathcal{M}_{K}$.  Then $\widetilde{\Pi}_{h}$ is stable on $T$, i.e.,
$
|\widetilde{\Pi}_{h}v|_{k,T} \lesssim h_{T}^{-k}\|v\|_{0,\Delta_{T}},\mbox{ with } k = 0, 1, 2.
$
\end{lemma}
\begin{proof}
We notice that $\Delta_{T}$ is at most a $5 \times 5$ patch. From the assumption of local quasi-uniformity in \eqref{eq:regularity}, we can conclude that all elements in $\Delta_{T}$ are of comparable size. Utilizing ~\eqref{eq:normEstimate}, we have $\max\limits_{ \substack{ \mathcal{M}_{K}\supset T \\ K\in \mathcal{J}_{h}} } |\varphi_{K} |_{k,T} \lesssim h_{T}^{1-k}$, where the hidden constant is only dependent on~$\gamma_{0}$.

\begin{align*}
|\widetilde{\Pi}_{h}v|_{k,T}^{2} & = \big|\sum_{ \substack{ \mathcal{M}_{K}\supset T \\ K\in \mathcal{J}_{h}} } \lambda_{K}(v) \varphi_{K}(x,y)\big|_{k,T}^{2} 
 \lesssim \sum_{ \substack{ \mathcal{M}_{K}\supset T \\ K\in \mathcal{J}_{h}} } |\lambda_{K}(v)|^{2} |\varphi_{K}(x,y)|_{k,T}^{2} \\ 
& \leq \max_{ \substack{ \mathcal{M}_{K}\supset T \\ K\in \mathcal{J}_{h}} } |\varphi_{K}(x,y)|_{k,T}^{2} \sum_{ \substack{ \mathcal{M}_{K}\supset T \\ K\in \mathcal{J}_{h}} } |\lambda_{K}(v)|^{2} 
\lesssim h_{T}^{2-2k} \sum_{ \substack{ \mathcal{M}_{K}\supset T \\ K\in \mathcal{J}_{h}} }\Big( \sum\limits_{\mu = 1}^{5}w_{\mu}^{K}\fint_{S_{\mu}^{K}}v{\,d x d y} \Big)^{2} \\
& \lesssim h_{T}^{2-2k} \sum_{ \substack{ \mathcal{M}_{K}\supset T \\ K\in \mathcal{J}_{h}} } \sum\limits_{\mu = 1}^{5}\big(\fint_{S_{\mu}^{K}}v{\,d x d y} \big)^{2}
 \lesssim h_{T}^{2-2k} \sum_{ \substack{ \mathcal{M}_{K}\supset T \\ K \in \mathcal{J}_{h}} } \sum\limits_{\mu = 1}^{5}\frac{1}{|S_{\mu}^{K}|}\int_{S_{\mu}^{K}}v^{2} {\,d x d y}  \\
& \lesssim h_{T}^{-2k} \|v\|_{0,\Delta_{T}}^{2}.
\end{align*}
The proof is thus completed. 
\end{proof}
\begin{lemma}\label{lem:approximationPi}
For any $T \in \mathcal{G}_{h}$, the following approximation property of  $\widetilde{\Pi}_{h}$ holds: 
\begin{equation}
|v-\widetilde{\Pi}_{h}v|_{k,T}\lesssim h_{T}^{s-k}|v|_{s,\Delta_{T}},\mbox{ with } 0\leq k \leq s \leq 3.
\end{equation}
\end{lemma}
\begin{proof}
For any polynomial $p \in P_{2}(\widetilde{\Omega}_{h})$, we have by Lemmas \ref{lem:preserveP2} and \ref{lem:stability} that 
\begin{align*}
|v-\widetilde{\Pi}_{h}v|_{k,T} & \leq |v-p|_{k,T} + |\widetilde{\Pi}_{h}(p-v)|_{k,T} \\
& \lesssim |v-p|_{k,T} + h_{T}^{-k}\|v-p\|_{0,\Delta_{T}} 
\end{align*}
Since $\Delta_{T}$ is a finite union of rectangles, each of which is star-shaped ensured by \eqref{eq:regularity}, we can apply the Bramble-Hilbert lemma in the form presented in \cite{Dupont;Scott1980,Scott;Zhang1990} and obtain 
\begin{align}\label{eq:bestP2}
\inf_{p\in P_{2}}|v-p|_{k,\Delta_{T}} \lesssim h_{T}^{s-k}|v|_{s,\Delta_{T}}, \mbox{ with }  0\leq k \leq s \leq 3,
\end{align}
where the hidden constant is only dependent on $\gamma_{0}$.
Therefore, we derive
\begin{align*}
|v-\widetilde{\Pi}_{h}v|_{k,T}\lesssim h_{T}^{s-k}|v|_{s,\Delta_{T}},\mbox{ with } 0\leq k \leq s \leq 3.
\end{align*}
The proof is completed. 
\end{proof}
\begin{lemma}{\rm(\!\cite[Theorem 1.4.5]{Brenner;Scott2007})}\label{lem:extension}
Suppose that $\Omega$ has a Lipschitz boundary. Then there is an extension mapping $E:~W^{p}_{k}(\Omega) \mapsto W^{p}_{k}(\mathbb{R}^{2})$ defined for all non-negative integers $k$ and real numbers $p$ in the range $1\leq p\leq \infty$ satisfying
 \begin{align}\label{eq:stabext}
 Ev\big|_{\Omega} = v, \quad 
 \|Ev\|_{W^{p}_{k}(\mathbb{R}^{2})} \leq C \|v\|_{W^{p}_{k}(\Omega)}, \quad \forall v \in W^{p}_{k}(\Omega),
 \end{align}
 where $C$ is a generic constant independent of $v$.
 \end{lemma}
\begin{theorem}\label{thm:approxOnDomain}
Let $E$ be an extension operator that satisfies \eqref{eq:stabext}. It holds for $v\in H^{s}(\Omega)$ that $|v-\widetilde{\Pi}_{h}E v|_{k,h}\lesssim h^{s-k}|v|_{s,\Omega}$ with $0\leq k \leq s \leq 3$.
\end{theorem}
\begin{proof}
Going through all elements in $\mathcal{G}_{h}$, by Lemmas~\ref{lem:approximationPi} and \ref{lem:extension}, we deduce the best approximation on $\Omega$, i.e., $|v-~\widetilde{\Pi}_{h}Ev|_{k,h} = |Ev-~\widetilde{\Pi}_{h}Ev|_{k,h} \lesssim h^{s-k}|Ev|_{s,\widetilde{\Omega}_{h}} \lesssim h^{s-k}|v|_{s,\Omega}$ with $0\leq k \leq s \leq 3.$
\end{proof}
These two lemmas are elementary but useful for verifying the approximation property of $\Pi_{h0}$; they can be found, for example, in \cite[Lemma~2]{Clement1975} and \cite[p24--p26]{Fichera2006}, respectively.
\begin{lemma}\label{lem:NormEquivalence} 
Let $e$ be an edge and $p\in P_{l}(e)$ with $l\geqslant 0$. Then 
$
|p|_{0,\infty, e}^{2} \lesssim |e|^{-1}|p|_{0,e}^{2}.
$
\end{lemma}
\begin{lemma}\label{lem:traceThm}
Let $K\in \mathcal{G}_{h}$, $e$ be an edge of $K$, and $v\in H^{1}(K)$. Then
$
|v|_{0,e}^{2}\lesssim h_{K}^{-1}|v|_{0,K}^{2} + h_{K}|v|_{1,K}^{2}.
$
\end{lemma}

Now we are going to prove Theorem \ref{thm:approxH02}.

\paragraph{\bf Proof of Theorem \ref{thm:approxH02}}
Let $E$ be an extension operator that satisfies \eqref{eq:bestP2}. Since $v-\Pi_{h0}v = (v - \widetilde{\Pi}_{h}Ev)+ (\widetilde{\Pi}_{h}Ev - \Pi_{h0}v)$, we only have to analyze $\widetilde{\Pi}_{h}Ev - \Pi_{h0}v$ cell by cell.
If $T\in \mathcal{G}_{h}$, and $\#\big\{\mathcal{M}_{K}: \ \mathcal{M}_{K}\cap \mathring{T} \ne \varnothing, K\in \mathcal{K}_{h}^{i}\big\} = 9$, then $(\widetilde{\Pi}_{h}Ev - \Pi_{h0}v)|_{T} = 0$, otherwise we have 
$$
(\widetilde{\Pi}_{h}Ev - \Pi_{h0}v)|_{T} = \sum\limits_{\substack{\mathcal{M}_{K}\supset T \\ K \in \mathcal{J}_{h} \backslash \mathcal{K}_{h}^{i}}} \lambda_{K}(Ev) \varphi_{K}|_{T}.
$$
First, we consider the case that $v\in H^{2}_{0}(\Omega)\cap H^{3}(\Omega)$. We insert some function $p\in P_{2}(\widetilde{\Omega}_{h})$ in the right-hand-side of the above equation. By \eqref{eq:normEstimate} and the proof procedure in Lemma \ref{lem:stability}, we obtain
\begin{align*}
\big|\sum\limits_{\substack{\mathcal{M}_{K}\supset T \\ K \in \mathcal{J}_{h} \backslash \mathcal{K}_{h}^{i}}} \lambda_{K}(Ev) \varphi_{K} \big|_{k,T}^{2} & = \big|\sum\limits_{\substack{\mathcal{M}_{K}\supset T \\ K \in \mathcal{J}_{h} \backslash \mathcal{K}_{h}^{i}}} \lambda_{K}(Ev-p) \varphi_{K} + \sum\limits_{\substack{\mathcal{M}_{K}\supset T \\ K \in \mathcal{J}_{h} \backslash \mathcal{K}_{h}^{i}}}\lambda_{K}(p) \varphi_{K}\ \big|_{k,T}^{2} \\
& \lesssim \big|\sum\limits_{\substack{\mathcal{M}_{K}\supset T \\ K \in \mathcal{J}_{h} \backslash \mathcal{K}_{h}^{i}}} \lambda_{K}(Ev-p) \varphi_{K} \  \big|_{k,T}^{2} + \big|\sum\limits_{\substack{\mathcal{M}_{K}\supset T \\ K \in \mathcal{J}_{h} \backslash \mathcal{K}_{h}^{i}}}\lambda_{K}(p) \varphi_{K}\  \big|_{k,T}^{2} \\
& \lesssim h_{T}^{-2k}\|Ev-p\|_{0,\Delta_{T}}^{2} + h_{T}^{2-2k}\sum\limits_{\substack{\mathcal{M}_{K}\supset T \\ K \in \mathcal{J}_{h} \backslash \mathcal{K}_{h}^{i}}} |\lambda_{K}(p) |^{2}.
\end{align*}

From Lemma~\ref{lem:property of phi} (b) and the construction of the functional $\lambda_{K}$, it holds that 
$$
\lambda_{K}(p) = p(c_{K}) - \frac{1}{8}\frac{\partial^{2}p}{\partial x^{2}}L_{K}^{2} - \frac{1}{8}\frac{\partial^{2}p}{\partial y^{2}}H_{K}^{2}.
$$

Thus, by the Taylor's expansion, there exists some $a_{K}\in \mathcal{N}_{h}^{b}$, $e_{K}\in \mathcal{E}_{h}^{b}$, and a boundary element $Q_{K}\in \mathcal{K}_{h}^{b}$, satisfying $a_{K}\in e_{K} \subset Q_{K}$, such that
\begin{align}\label{eq:lambdaP}
\lambda_{K}(p) = p(a_{K}) + (-1)^{\delta_{1}} \frac{\partial p }{\partial x}(a_{K})\frac{L_{K}}{2}+(-1)^{\delta_{2}}\frac{\partial p }{\partial y}(a_{K})\frac{H_{K}}{2} + (-1)^{\delta_{1}+\delta_{2}}\frac{\partial^{2} p }{\partial x \partial y}\frac{L_{K}H_{K}}{4},
\end{align}
where  $\delta_{1}$ and $\delta_{2}$ equals to $\pm 1$, and their specific values are determined by the relative position of $c_{K}$ and $a_{K}$. Since $v \in H^{2}_{0}(\Omega)$, it can be deduced that 
\begin{align}\label{eq:uH02}
|v|_{0,e_{K}} =\big|\tfrac{\partial v}{\partial x}\big|_{0,e_{K}} = \big|\tfrac{\partial v}{\partial y}\big|_{0,e_{K}} = \big|\tfrac{\partial^{2}v}{\partial x \partial y}\big|_{0,e_{K}} = 0.
\end{align}
From Lemma \ref{lem:traceThm} and \eqref{eq:uH02},  we have 
\begin{equation}\label{eq:traceP}
\begin{split}
& |p|_{0,e_{K}}^{2} = |v-p|_{0,e_{K}}^{2} \lesssim h_{Q_{K}}^{-1} |v-p|_{0,Q_{K}}^{2} + h_{Q_{K}} |v-p|_{1,Q_{K}}^{2}; \\
& \big|\tfrac{\partial p}{\partial x}\big|_{0,e_{K}}^{2} = \big|\tfrac{\partial v}{\partial x} -\tfrac{\partial p}{\partial x}\big|_{0,e_{K}}^{2} \lesssim h_{Q_{K}}^{-1} |v-p|_{1,Q_{K}}^{2} + h_{Q_{K}}|v-p|_{2,Q_{K}}^{2}; \\
& \big|\tfrac{\partial p}{\partial y}\big|_{0,e_{K}}^{2} = \big|\tfrac{\partial v}{\partial y} -\tfrac{\partial p}{\partial y}\big|_{0,e_{K}}^{2} \lesssim h_{Q_{K}}^{-1} |v-p|_{1,Q_{K}}^{2} + h_{Q_{K}} |v-p|_{2,Q_{K}}^{2}; \\
& \big|\tfrac{\partial^{2} p}{\partial x \partial y}\big|_{0,e_{K}}^{2} = \big|\tfrac{\partial^{2} v}{\partial x \partial y} -\tfrac{\partial^{2} p}{\partial x \partial y}\big|_{0,e_{K}}^{2} \lesssim h_{Q_{K}}^{-1} |v-p|_{2,Q_{K}}^{2} + h_{Q_{K}}  |v-p|_{3,Q_{K}}^{2}.
\end{split}
\end{equation}
A combination of Lemma \ref{lem:NormEquivalence},  \eqref{eq:bestP2}, and \eqref{eq:traceP} leads to 
\begin{align}
h_{T}^{-2k}\|Ev-p\|_{0,\Delta_{T}}^{2} + h_{T}^{2-2k}\sum\limits_{\substack{\mathcal{M}_{K}\supset T \\ K \in \mathcal{J}_{h} \backslash \mathcal{K}_{h}^{i}}} |\lambda_{K}(p)|^{2} \lesssim h_{T}^{2(3-k)} |Ev|_{3,\Delta_{T}}, \mbox{ with } 0\leqslant k \leqslant 3.
\end{align}
For the case of a lower regularity that $v \in H^{2}_{0}(\Omega)$, we assume $p\in P_{1}(\widetilde{\Omega}_{h})$, and then $\lambda_{K}(p) = p(c_{K}) = p(a_{K}) + (-1)^{\delta_{1}} \frac{\partial p }{\partial x}(a_{K})\frac{L_{K}}{2}+(-1)^{\delta_{2}}\frac{\partial p }{\partial y}(a_{K})\frac{H_{K}}{2}$. For the case that $v \in H^{1}_{0}(\Omega)$, we utilize some $p\in P_{0}(\widetilde{\Omega}_{h})$, and then $\lambda_{K}(p) = p(c_{K}) = p(a_{K})$. By repeating the above process, similar results can be obtained for those two cases. Finally we have
$|v-\Pi_{h0}v|_{k,h}\lesssim h^{s-k}|v|_{s,\Omega}, \mbox{ with } 0\leq k \leq s \leq 3.$ The proof is completed.
\qed
\begin{remark}
{\rm
We note that the operator $\Pi_{h0}$ is not a projection. Actually, with the given basis functions, no locally-defined interpolation can be projective. We refer to the appendix for detailed discussions. 
}
\end{remark}
\section{A robust optimal scheme for the model problem}
\label{sec:convergence}
Associated with the the RRM element space $V_{h0}^{\rm{R}}$, a finite element scheme for \eqref{eq:model problem} is defined as:   Find $u_{h}^{\rm{R}}\in V_{h0}^{\rm{R}}$, such that
\begin{align}\label{eq:discrete form RRM}
\varepsilon^{2}a_{h}(u_{h}^{\rm{R}},v_{h}) + b_{h}(u_{h}^{\rm{R}},v_{h}) = (f,v_{h}), \quad \forall v_{h} \in V_{h0}^{\rm{R}}.
\end{align}
Define $ \interleave  w \interleave_{\varepsilon,h} := \sqrt{\varepsilon^{2} a_{h}(w,w)+b_{h}(w,w)}$, then $\interleave  \cdot \interleave_{\varepsilon,h}$ is a norm on $V+V_{h0}^{\rm{R}}$. The well-posedness of \eqref{eq:discrete form RRM} follows by the Lax-Milgram lemma. 
\begin{theorem}\label{thm:errorRRM}
Let $\Omega \subset \mathbb{R}^{2}$ be a bounded domain and $\mathcal{G}_{h}$ be a regular rectangular subdivision which covers $\Omega$. Let $u$ and $u_{h}^{\rm R}$ be the solutions of \eqref{eq:model problem} and \eqref{eq:discrete form RRM}, respectively. There exists a constant $C$, uniform with respect to $\varepsilon$ and $h$, such that it holds when $u \in H^{2}_{0}(\Omega) \cap H^{3}(\Omega)$ that 
\begin{equation}
\interleave u-u_{h}^{\rm{R}} \interleave_{\varepsilon,h} \leqslant C\left[ h |u|_{2,\Omega} + \varepsilon h |u|_{3,\Omega}  + \varepsilon^{2} h \|\Delta^{2}u\|_{0,\Omega}\right].
\end{equation}
If further $\mathcal{G}_{h}$ is uniform, then 
\begin{equation}\interleave u-u_{h}^{\rm{R}} \interleave_{\varepsilon,h} \leqslant C\left[ h^{2}|u|_{3,\Omega} + \varepsilon h |u|_{3,\Omega}  +   \varepsilon^{2} h \|\Delta^{2}u\|_{0,\Omega}\right].
\end{equation}
\end{theorem}
\begin{proof}
By the Strang's second lemma \cite{P.G.Ciarlet2002}, we have
\begin{align}\label{eq:StrangRRM}
\interleave u-u_{h}^{\rm R} \interleave_{\varepsilon,h} \lesssim \inf_{v_{h}\in V_{h0}^{\rm{R}}} \interleave u - v_{h} \interleave_{\varepsilon,h} + \sup_{w_{h}\in V_{h0}^{\rm{R}},w_{h} \ne 0}  \frac{E_{\varepsilon,h}(u,w_{h})}{\interleave w_{h}\interleave_{\varepsilon,h} },
\end{align}
where $E_{\varepsilon,h}(u,w_{h}) = \varepsilon^{2} a_{h}(u,w_{h}) + b_{h}(u,w_{h}) -(f,w_{h})$.

For the approximation error, by Theorem \ref{thm:approxH02}, it holds  for $w\in  H^{2}_{0}(\Omega)\cap H^{3}(\Omega)$ that
\begin{equation}\label{lem:approx in energy norm}
\inf_{v_{h}\in V_{h0}^{\rm{R}}} \interleave w - v_{h} \interleave_{\varepsilon,h} \leq \interleave w - \Pi_{h0}w \interleave_{\varepsilon,h} \lesssim 
\left\{
\begin{array}{l}
h(|w|_{2,\Omega} +\varepsilon |w|_{3,\Omega}), \\
h(h+\varepsilon)|w|_{3,\Omega}.
\end{array}
\right.
\end{equation}

For the consistency error, let $\Pi_{h0}^{\rm{b}}$ be the nodal interpolation operator associated with the bilinear element, then
\begin{equation}
\begin{split}
E_{\varepsilon,h}(u,w_{h}) & =  \varepsilon^{2}a_{h}(u,w_{h}) + b_{h}(u,w_{h}) - (f,w_{h}) \\
 & = \varepsilon^{2} \sum_{K\in \mathcal{G}_{h}} \int_{K} \nabla^{2}u : \nabla^{2}w_{h}  + \sum_{K\in \mathcal{G}_{h}} \int_{K} \nabla u \cdot \nabla w_{h}  - \int_{\Omega}(\varepsilon^{2} \Delta^{2} u - \Delta u)w_{h} \\
 & = \varepsilon^{2} \sum_{K\in \mathcal{G}_{h}} \int_{K} (\nabla^{2}u : \nabla^{2}w_{h} - \Delta^{2} u \ \Pi_{h0}^{\rm{b}} w_{h} ) + \sum_{K\in \mathcal{G}_{h}} \int_{K} (\nabla u \cdot \nabla w_{h} + \Delta u \ w_{h}) \\
&  \quad + \varepsilon^{2}\int_{\Omega} \Delta^{2} u\ (\Pi_{h0}^{\rm{b}} w_{h} - w_{h}) := R_{1} + R_{2} +R_{3}.
\end{split}
\end{equation}
We denote $\partial_{1}: = \partial_{x}$, $\partial_{2}: = \partial_{y}$, and utilize similar abbreviations for higher derivatives. Since $\Pi_{h0}^{\rm{b}} w_{h}\in H^{1}_{0}(\Omega)$, we have, by the Green's formula,
\begin{equation}
R_{1} = \varepsilon^{2} \sum_{K\in \mathcal{G}_{h}}  \sum_{i,j = 1:2} \int_{K} \partial_{iij}u \ \partial_{j}(\Pi_{h0}^{\rm{b}} w_{h} - w_{h}) + \varepsilon^{2}\sum_{K\in \mathcal{G}_{h}}\sum_{i,j = 1:2} \int_{\partial_{K}} \partial_{ij}u \ \partial_{j} w_{h} \ n_{i}\,d s.
\end{equation}
By the Cauchy-Schwarz inequality and the approximation property of the interpolation $\Pi_{h0}^{\rm{b}} $, 
\begin{equation}\label{eq:R1_1}
\Big| \varepsilon^{2} \sum_{K\in \mathcal{G}_{h}}  \sum_{i,j = 1:2} \int_{K} \partial_{iij}u \ \partial_{j}(\Pi_{h0}^{\rm{b}} w_{h} - w_{h}) \Big| \lesssim
\left\{
\begin{array}{l}
 \varepsilon h |u|_{3,\Omega} \interleave w_{h} \interleave_{\varepsilon,h}, \\
 \varepsilon^{\frac{3}{2}} h^{\frac{1}{2}} |u|_{3,\Omega} \interleave w_{h} \interleave_{\varepsilon,h}.
\end{array}
\right.
\end{equation}
 Note that $ \fint_{F} \llbracket\partial_{(i,j)} w_{h} \rrbracket_F \,d s = 0, i = 1:2, \forall F\in \mathcal{E}_{h}$. Let $R_{F}^{0}$ be the $L^{2}$ projection operator onto $P_{0}(F)$.
 \begin{equation}\label{eq:R1_2}
\begin{split}
\big| \varepsilon^{2}\sum_{K\in \mathcal{G}_{h}}\sum_{i,j = 1:2} \int_{\partial_{K}} \partial_{ij}u \ \partial_{j} w_{h} \ n_{i}\,d s \big| & = \big| \varepsilon^{2}\sum_{K\in \mathcal{G}_{h}}\sum_{i,j = 1:2} \sum_{F\in \partial_{K}}\int_{F} \partial_{ij}u \ \partial_{j} w_{h}\ (n_{F})_{i}\,d s \big| \\
& = \big| \varepsilon^{2}\sum_{K\in \mathcal{G}_{h}}\sum_{i,j = 1:2} \sum_{F\in \partial_{K}}\int_{F} R_{F}^{0} (\partial_{ij}u) \ R_{F}^{0} (\partial_{j}  w_{h}) \ (n_{F})_{i}\,d s \big| \\
& \lesssim \varepsilon^{2}\sum_{K\in \mathcal{G}_{h}}\sum_{i,j = 1:2} \sum_{F\in \partial_{K}}  \| R_{F}^{0} (\partial_{ij}u) \|_{0,F}  \| R_{F}^{0} (\partial_{j}  w_{h})  \|_{0,F} \\
& \lesssim  
\left\{
\begin{array}{l}
\varepsilon h |u|_{3,\Omega} \interleave w_{h} \interleave_{\varepsilon,h}, \\
\varepsilon h^{\frac{1}{2}} |u|_{2,\Omega}^{\frac{1}{2}}|u|_{3,\Omega}^{\frac{1}{2}} \interleave w_{h} \interleave_{\varepsilon,h}.
\end{array}
\right.
\end{split}
\end{equation}
Summing up \eqref{eq:R1_1} and \eqref{eq:R1_2}, we obtain that 
\begin{equation}\label{eq:R1}
R_{1} \lesssim
\left\{
\begin{array}{l}
\varepsilon h |u|_{3,\Omega} \interleave w_{h} \interleave_{\varepsilon,h}, \\
\big(\varepsilon^{\frac{3}{2}} h^{\frac{1}{2}} |u|_{3,\Omega} + \varepsilon h^{\frac{1}{2}} |u|_{2,\Omega}^{\frac{1}{2}}|u|_{3,\Omega}^{\frac{1}{2}} \big)\interleave w_{h} \interleave_{\varepsilon,h} .
\end{array}
\right.
\end{equation}
From $V_{h0}^{\rm{R}} \subset V_{hs}^{\rm{M}}$ and $H^{2}_{0}(\Omega) \subset H^{1}_{0}(\Omega)$, we have $R_{2} \lesssim h|u|_{2,\Omega}\interleave w_{h}\interleave_{\varepsilon,h}$ by Lemma \ref{lem:consisRM}. Specially, if the mesh is uniform, then $R_{2} \lesssim h^{2}|u|_{3,\Omega}\interleave w_{h}\interleave_{\varepsilon,h}$.
For $R_{3}$, it holds that $R_{3}\lesssim \varepsilon^{2} h  \|\Delta^{2}u\|_{0,\Omega}\interleave w_{h}\interleave_{\varepsilon,h}$.
By \eqref{eq:StrangRRM}, \eqref{lem:approx in energy norm}, together with the estimates of terms $R_{1}$, $R_{2}$, and $R_{3}$, the results can be derived immediately.
\end{proof}

It appears in Theorem \ref{thm:errorRRM} that, the RRM element, which is a nonconforming quadratic finite element scheme, ensures linear convergence with respect to $h$, uniformly in $ \varepsilon$, as long as the term $|u|_{2,\Omega}  + \varepsilon |u|_{3,\Omega} + \varepsilon^{2} \|\Delta^{2}u\|_{0,\Omega}$ is uniformly bounded. When $\varepsilon$ approaches zero, the convergence rate of this scheme approaches $O(h^{2})$ in the energy norm on uniform grids, provided that the solution is sufficiently smooth.  As is mentioned in \cite{Nilssen;Tai;Winther2001}, it may happen that $|u|_{2,\Omega}$ and $|u|_{3,\Omega}$ blow up when $\varepsilon \rightarrow 0$. By the regularity result on a convex domain in Lemma \ref{lem:regularity}, we conclude with the following uniform convergence property for the RRM element method.
\begin{theorem}\label{thm:uniform convergence}
Let $\Omega$ be convex and $f\in L^{2}(\Omega)$. Then 
$\interleave u-u_{h}^{\rm{R}}\interleave_{\varepsilon,h}\lesssim h^{\frac{1}{2}}\|f\|_{0,\Omega}$.
\end{theorem}
\begin{proof}
From Theorem~\ref{thm:approxH02}, we have
$\big|u-\Pi_{h0}u\big|_{2,h}^{2} \lesssim |u|_{2,\Omega}\big|u-\Pi_{h0}u\big|_{2,h}  \lesssim h |u|_{2,\Omega} |u|_{3,\Omega}.$ By Lemma ~\ref{lem:regularity}, we further obtain
\begin{align}\label{eq:H2norm(u-Piuh)}
\varepsilon^{2}\big|u-\Pi_{h0}u\big|_{2,h}^{2} \lesssim h \varepsilon^{\frac{1}{2}}|u|_{2,\Omega} \varepsilon^{\frac{3}{2}}|u|_{3,\Omega}  \lesssim h \|f\|_{0,\Omega}^{2}.
\end{align}
From \cite[Theorem 3.2.1.2] {Grisvard1985}, $|u^{0}|_{2,\Omega} \lesssim \|f\|_{0,\Omega}$. This, together with Lemma ~\ref{lem:regularity}, leads to 
\begin{equation}\label{eq:H1norm(u-Piuh)}
\begin{split}
\big|u-\Pi_{h0}u\big|_{1,h}^{2}  & \lesssim \big|u-u^{0}- \Pi_{h0}(u-u^{0})\big|_{1,h}^{2} + \big|u^{0}-\Pi_{h0}u^{0}\big|_{1,h}^{2} \\
 & \lesssim \big|u-u^{0}\big|_{1,\Omega}h\big|u-u^{0}\big|_{2,\Omega} + h^{2}\big|u^{0}\big|_{2,\Omega}^{2} \\
 & \lesssim h(\varepsilon^{-\frac{1}{2}} \big|u-u^{0}\big|_{1,\Omega})(\varepsilon^{\frac{1}{2}} \big|u-u^{0}\big|_{2,\Omega}) + h^{2}\big|u^{0}\big|_{2,\Omega}^{2} \\
 & \lesssim h\|f\|_{0,\Omega}^{2} + h^{2}\|f\|_{0,\Omega}^{2}  \lesssim h\|f\|_{0,\Omega}^{2}.
 \end{split}
\end{equation}
From \eqref{eq:H2norm(u-Piuh)} and \eqref{eq:H1norm(u-Piuh)}, we obtain
\begin{align}\label{eq:unifrom approx error}
\inf_{v_{h}\in V_{h0}^{\rm{R}}} \interleave u - v_{h} \interleave_{\varepsilon,h} \leq \interleave u - \Pi_{h0}u \interleave_{\varepsilon,h} \lesssim h^{\frac{1}{2}}\|f\|_{0,\Omega}.
\end{align}
Owing to the second estimate in \eqref{eq:R1}, it yields that
$$
R_{1} \lesssim  \big(\varepsilon^{\frac{3}{2}} h^{\frac{1}{2}} |u|_{3,\Omega} + \varepsilon h^{\frac{1}{2}} |u|_{2,\Omega}^{\frac{1}{2}}|u|_{3,\Omega}^{\frac{1}{2}} \big)\interleave w_{h} \interleave_{\varepsilon,h} \lesssim h^{\frac{1}{2}}\|f\|_{0,\Omega}\interleave w_{h} \interleave_{\varepsilon,h}.
$$
By \eqref{eq:model problem} and \eqref{eq:2nd problem},  $\varepsilon^{2} \Delta^{2}u = \Delta(u-u^{0})$.
When $h< \varepsilon$, we have, by Lemmas \ref{lem:regularity} and \ref{lem:consisRM}, that
\begin{equation*}
\begin{split}
R_{2}+R_{3} & =  \sum_{K\in \mathcal{G}_{h}} \int_{K} (\nabla u \cdot \nabla w_{h} + \Delta u \ w_{h}) + \varepsilon^{2}\int_{\Omega} \Delta^{2} u \ (\Pi_{h0}^{\rm{b}} w_{h} - w_{h})  \\
& =  \sum_{K\in \mathcal{G}_{h}} \int_{K} (\nabla u \cdot \nabla w_{h} + \Delta u \ w_{h}) + \int_{\Omega} \Delta(u-u^{0}) (\Pi_{h0}^{\rm{b}} w_{h} - w_{h}) \\
& \lesssim h|u|_{2,\Omega}|w_{h}|_{1,h} + h|u-u^{0}|_{2,\Omega}|w_{h}|_{1,h} \\
& \lesssim h^{\frac{1}{2}} \varepsilon^{\frac{1}{2}}|u|_{2,\Omega}|w_{h} |_{1,h} + h^{\frac{1}{2}} \varepsilon^{\frac{1}{2}}(|u|_{2,\Omega}+|u^{0}|_{2,\Omega}) |w_{h} |_{1,h}   \lesssim h^{\frac{1}{2}}\|f\|_{0,\Omega}\interleave w_{h} \interleave_{\varepsilon,h}.
\end{split}
\end{equation*}
When $\varepsilon \leq h$, noticing that $\Pi_{h0}^{\rm{b}} w_{h} \in H^{1}_{0}(\Omega)$, we obtain
\begin{equation*}
\begin{split}
R_{2}+R_{3} & =  \sum_{K\in \mathcal{G}_{h}} \int_{K} (\nabla u \cdot \nabla w_{h} + \Delta u \ w_{h}) + \int_{\Omega} \Delta(u-u^{0})(\Pi_{h0}^{\rm{b}} w_{h} - w_{h})  \\
& =  \sum_{K\in \mathcal{G}_{h}} \int_{K} (\nabla u \cdot \nabla (w_{h}-\Pi_{h0}^{\rm{b}} w_{h}) + \Delta u (w_{h}-\Pi_{h0}^{\rm{b}} w_{h})) - \int_{\Omega} \Delta(u-u^{0})(w_{h} -\Pi_{h0}^{\rm{b}} w_{h} )  \\
& =  \sum_{K\in \mathcal{G}_{h}} \int_{K} \nabla u \cdot \nabla (w_{h}-\Pi_{h0}^{\rm{b}} w_{h})  + \int_{\Omega} \Delta u^{0} (w_{h} -\Pi_{h0}^{\rm{b}} w_{h} ) \\
& = \sum_{K\in \mathcal{G}_{h}} \int_{K} \nabla (u-u^{0}) \cdot \nabla (w_{h}-\Pi_{h0}^{\rm{b}} w_{h}) + \sum_{K\in \mathcal{G}_{h}} \int_{\partial K} \frac{\partial u^{0}}{\partial \mathbf{n}} (w_{h}-\Pi_{h0}^{\rm{b}} w_{h}) \\
& \lesssim \varepsilon^{-\frac{1}{2}} |u-u^{0}|_{1,\Omega}  \varepsilon^{\frac{1}{2}} |w_{h}|_{1,h} + h|u^{0}|_{2,\Omega}|w_{h}|_{1,h} \\
& \lesssim \|f\|_{0,\Omega}h^{\frac{1}{2}} |w_{h}|_{1,h}  + h\|f\|_{0,\Omega} |w_{h}|_{1,h} \lesssim h^{\frac{1}{2}}\|f\|_{0,\Omega} \interleave w_{h} \interleave_{\varepsilon,h},
\end{split}
\end{equation*}
where Lemma~\ref{lem:consisRM} is utilized to estimate the term $\sum_{K\in \mathcal{G}_{h}} \int_{\partial K} \frac{\partial u^{0}}{\partial \mathbf{n}} (w_{h}-\Pi_{h0}^{\rm{b}} w_{h})$.
Hence we obtain that
\begin{align}\label{eq:unifrom consis error}
E_{\varepsilon,h}(u,w_{h}) = R_{1}+R_{2}+R_{3} \lesssim h^{\frac{1}{2}}\|f\|_{0,\Omega} \interleave w_{h} \interleave_{\varepsilon,h}.
\end{align}
Combining \eqref{eq:StrangRRM}, \eqref{eq:unifrom approx error}, and \eqref{eq:unifrom consis error}, the uniform estimate is finally obtained.
\end{proof}

\section{Numerical experiments}
\label{sec:numerical}
We consider both uniform subdivisions and non-uniform subdivisions of $\Omega$; see~Figure~\ref{fig:non-uniformgrid}. Numerical examples of the model problem~\eqref{eq:model problem} are given below. 
\begin{figure}[!htbp]
\centering
\includegraphics[height=0.29\hsize]{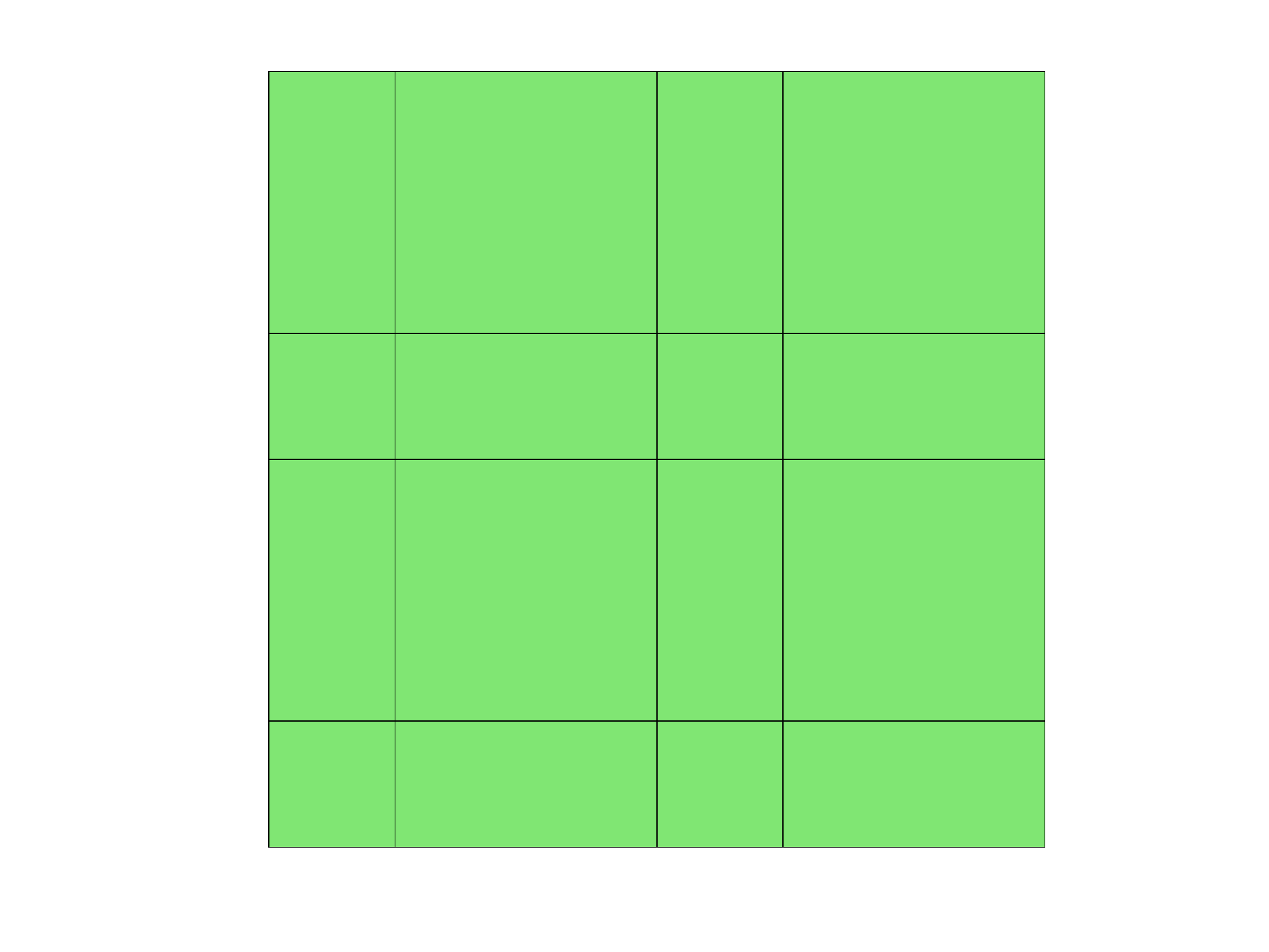}
\qquad
\includegraphics[height=0.29\hsize]{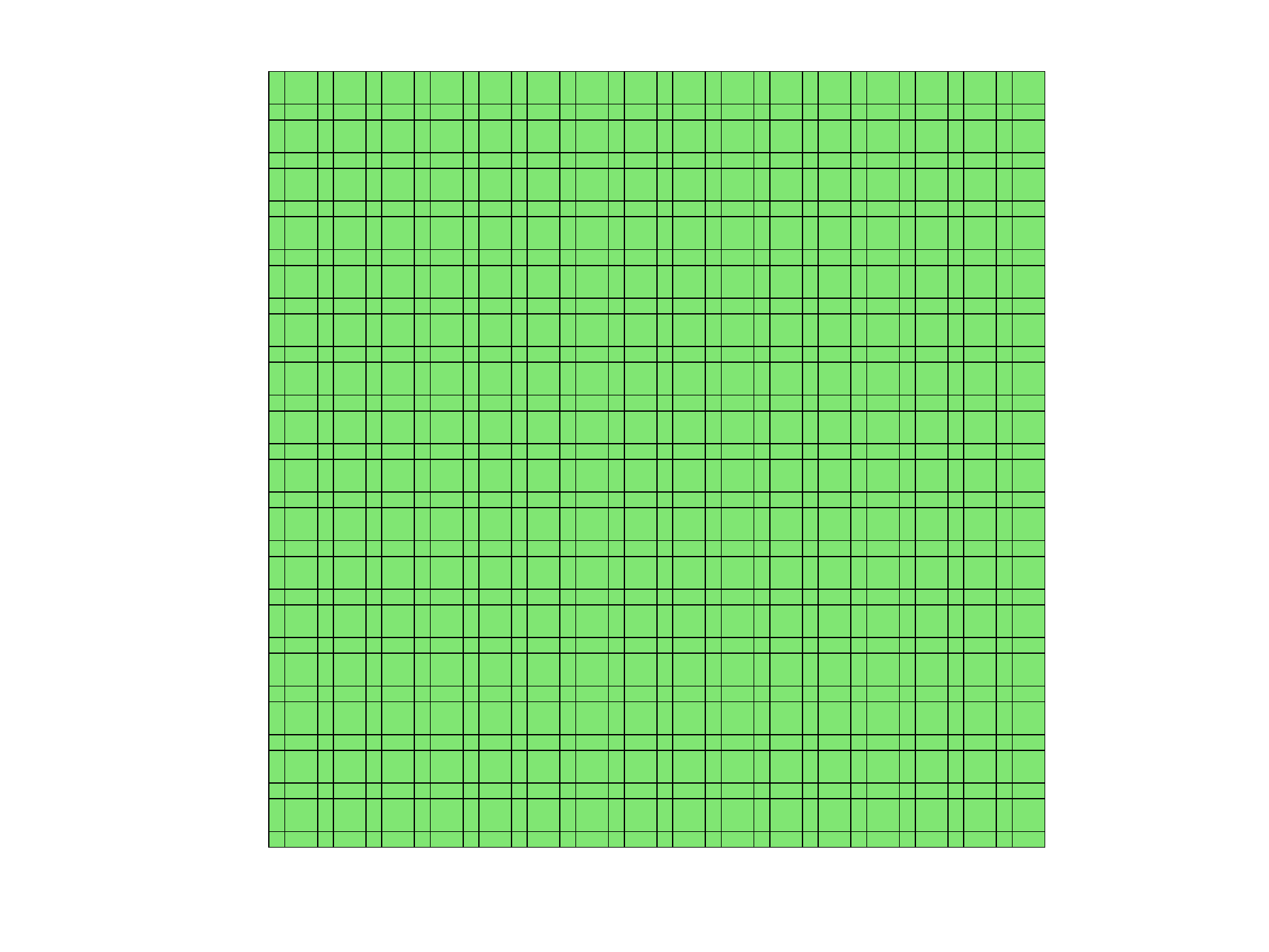}
\caption{Illustration of a non-uniform shape regular subdivision. The partition in the right is a combination of small patterns as the left one.}\label{fig:non-uniformgrid}
\end{figure}

\vspace{0.1cm}
\noindent{\bf Example 1}
 
\noindent Let $\Omega=(0,1)^{2}$. Take $u = \big(sin(\pi x)sin(\pi y)\big)^{2}$, and set $f = \varepsilon^{2} \Delta^{2}u - \Delta u$. Then $u$ is the solution of problem~\eqref{eq:model problem} when $\varepsilon > ~0$. Apply \eqref{eq:discrete form RRM}  to get the discrete solution $u_{h}^{\rm{R}}$ on uniform or non-uniform meshes, and compute the relative error $\frac{\interleave u-u_{h}^{\rm{R}} \interleave_{\varepsilon,h}}{\interleave u \interleave_{\varepsilon,h}}$. From the experiment result of the non-uniform case shown in Table~\ref{tab:non-uniform Example1}, the convergence rate is $\mathcal{O}(h)$ for $0<\varepsilon \leq1$. From the experiment result of the uniform case shown in Table~\ref{tab:uniform Example1}, the convergence rate is $\mathcal{O}(h)$ when $\varepsilon = \mathcal{O}(1)$ and $\mathcal{O}(h^{2})$ when $\varepsilon \ll 1 $. Both cases verify the theoretical findings in Theorem \ref{thm:errorRRM}.

\vspace{0.1cm}
\noindent{\bf Example 2}: 

Let $\Omega = (0,2)^{2}\backslash [1,2]^{2}$. Take the same $u$ as in {\bf Example 1}. From Tables~\ref{tab:non-uniform Example2} and \ref{tab:uniform Example2}, the convergence rates on the L-shaped domain are consistent with the results derived on $\Omega = (0,1)^{2}$.  It verify the theoretical findings, especially the results in Theorem~\ref{thm:approxH02}, which shows that the interpolating properties are valid for non-convex domains. 

\vspace{0.1cm}
\noindent{\bf Example 3}
 
 \noindent  Let $\Omega=(0,1)^{2}$. Consider \eqref{eq:model problem} with $f = 2\pi^2sin(\pi x)sin(\pi y)$. It is proposed in \cite{Wang;Huang;Tang;Zhou2018} that the exact solution with this right hand term possesses strong boundary layers when $\varepsilon$ is very small. The explict expression of $u$ is unknown, but the exact solution of \eqref{eq:2nd problem} reads $u^{0} = sin(\pi x)sin(\pi y)$. From \cite[(3.29)]{Wang;Huang;Tang;Zhou2018}, it holds that $\interleave u^{0}-u_{h}^{\rm{R}} \interleave_{\varepsilon,h} \lesssim (\varepsilon^{\frac{1}{2}}+h^{\frac{1}{2}}) \|f\|_{0,\Omega}$ under the same assumptions in Theorem \ref{thm:uniform convergence}. Here we take $\varepsilon$ to be small enough, such that $\varepsilon < h$. From Tables~\ref{tab:non-uniform Example3} and~\ref{tab:uniform Example3}, the convergence rate of the relative error $\frac{\interleave u^{0}-u_{h}^{\rm{R}} \interleave_{\varepsilon,h}}{\interleave u^{0} \interleave_{\varepsilon,h}}$ is $\mathcal{O}(h^{\frac{1}{2}})$, which verify the uniform convergence of the RRM element in Theorem~\ref{thm:uniform convergence}.
\begin{table}[h!!]
	\caption{Relative error $\frac{\interleave u-u_{h}^{\rm{R}} \interleave_{\varepsilon,h}}{\interleave u \interleave_{\varepsilon,h}}$ on non-uniform grids in Example 1.}
	\label{tab:non-uniform Example1}
	\begin{tabular}{lllllll}
		\hline\noalign{\smallskip}	
		$\varepsilon \setminus h$ & 3.250e-1 & 1.625e-1 & 8.125e-2  &  4.063e-2  & 2.031e-2 & Rate \\
		\noalign{\smallskip}\hline\noalign{\smallskip}
		$2^{0}$ & 0.6196  & 0.3127 & 0.1556  & 0.0776 & 0.0388 & 1.00 \\
		$2^{-2}$ & 0.5691 & 0.2798 & 0.1380 & 0.0686 & 0.0343 & 1.01 \\
		$2^{-4}$ & 0.3691 & 0.1597 & 0.0699 & 0.0330 & 0.0162 & 1.13 \\
		$2^{-6}$ & 0.2825 &	0.1318 & 0.0553  & 0.0192 & 0.0064 & 1.37 \\
		$2^{-8}$ & 0.2746 & 0.1337 & 0.0664  & 0.0311 & 0.0127 & 1.10 \\
		$2^{-10}$ & 0.2741 & 0.1339 & 0.0676 & 0.0338 & 0.0166 & 1.01 \\
		\noalign{\smallskip}\hline
	\end{tabular}
\end{table}
\begin{table}[h!!]
\caption{Relative error $\frac{\interleave u-u_{h}^{\rm{R}} \interleave_{\varepsilon,h}}{\interleave u \interleave_{\varepsilon,h}}$ on uniform grids in Example 1.}
\label{tab:uniform Example1}
\begin{tabular}{lllllll}
\hline\noalign{\smallskip}
$\varepsilon \setminus h$ & $2^{-2}$ & $2^{-3}$ &  $2^{-4}$  &  $2^{-5}$  & $2^{-6}$ & Rate \\
\noalign{\smallskip}\hline\noalign{\smallskip}
$2^{0}$ & 0.5403  & 0.2754 & 0.1376  & 0.0688 & 0.0344 & 0.99 \\
$2^{-2}$ & 0.4890 & 0.2448 & 0.1218 & 0.0608 & 0.0304 & 1.00 \\
$2^{-4}$ & 0.2926 & 0.1238 & 0.0585 & 0.0288 & 0.0144 & 1.08 \\
$2^{-6}$ & 0.2080 & 0.0585 & 0.0199  & 0.0084 & 0.0040 & 1.42 \\
$2^{-8}$ & 0.2002 & 0.0502 & 0.0130  & 0.0037 & 0.0013 & 1.84 \\
$2^{-10}$ & 0.1996 & 0.0496 & 0.0124  &	0.0031 & 0.0008 & 1.99 \\
\noalign{\smallskip}\hline
 \end{tabular}
\end{table}
\begin{table}[h!!]
	\caption{Relative error $\frac{\interleave u-u_{h}^{\rm{R}} \interleave_{\varepsilon,h}}{\interleave u \interleave_{\varepsilon,h}}$ on non-uniform grids in Example 2.}
	\label{tab:non-uniform Example2}
	\begin{tabular}{lllllll}
		\hline\noalign{\smallskip}	
		$\varepsilon \setminus h$ & 3.250e-1 & 1.625e-1 & 8.125e-2  &  4.063e-2  & 2.031e-2 & Rate \\
		\noalign{\smallskip}\hline\noalign{\smallskip}
		$2^{0}$ & 0.6236  & 0.3142 & 0.1558  & 0.0776 & 0.0388 & 1.00 \\
		$2^{-2}$ & 0.5722 & 0.2812 & 0.1382 & 0.0687 & 0.0343 & 1.02 \\
		$2^{-4}$ & 0.3711  & 0.1610 & 0.0700 & 0.0330 & 0.0162 & 1.13 \\
		$2^{-6}$ & 0.2863 &	0.1349 & 0.0556  & 0.0192 & 0.0064 & 1.38 \\
		$2^{-8}$ & 0.2787 & 0.1373 & 0.0668  & 0.0312 & 0.0127 & 1.11 \\
		$2^{-10}$ & 0.2782 & 0.1375 & 0.0681 & 0.0338 & 0.0166 & 1.02 \\
		\noalign{\smallskip}\hline
	\end{tabular}
\end{table}
\begin{table}[h!!]
\caption{Relative error $\frac{\interleave u-u_{h}^{\rm{R}} \interleave_{\varepsilon,h}}{\interleave u \interleave_{\varepsilon,h}}$ on uniform grids in Example 2.}
\label{tab:uniform Example2}
\begin{tabular}{lllllll}
\hline\noalign{\smallskip}
$\varepsilon \setminus h$ & $2^{-2}$ & $2^{-3}$ &  $2^{-4}$  &  $2^{-5}$  & $2^{-6}$ & Rate \\
\noalign{\smallskip}\hline\noalign{\smallskip}
$2^{0}$ & 0.5463  & 0.2763 & 0.1377  & 0.0688 & 0.0344 & 1.00 \\
$2^{-2}$ & 0.4938 & 0.2456 & 0.1219 & 0.0608 & 0.0304 & 1.01 \\
$2^{-4}$ & 0.2937 & 0.1242 & 0.0586 & 0.0288 & 0.0144 & 1.08 \\
$2^{-6}$ & 0.2077 & 0.0585 & 0.0200  & 0.0084 & 0.0040 & 1.42 \\
$2^{-8}$ & 0.1997 & 0.0502 & 0.0130  & 0.0037 & 0.0013 & 1.84 \\
$2^{-10}$ & 0.1991 & 0.0496 & 0.0124  &	0.0031 & 0.0008 & 1.98 \\
\noalign{\smallskip}\hline
 \end{tabular}
\end{table}
 \begin{table}[h!!]
 	\caption{Relative error $\frac{\interleave u^{0}-u_{h}^{\rm{R}} \interleave_{\varepsilon,h}}{\interleave u^{0} \interleave_{\varepsilon,h}}$ on non-uniform grids in Example 3.}
 	\label{tab:non-uniform Example3}
 	\begin{tabular}{lllllll}
 		\hline\noalign{\smallskip}
		$\varepsilon \setminus h$ & 3.250e-1 & 1.625e-1 & 8.125e-2  &  4.063e-2  & 2.031e-2 & Rate \\
 		\noalign{\smallskip}\hline\noalign{\smallskip}
 		$2^{-8}$ & 0.5846 & 0.3945 & 0.2755  & 0.1995 & 0.1555 & 0.48 \\
 		$2^{-10}$ & 0.5843 & 0.3934 & 0.2725  & 0.1912 & 0.1358 & 0.53 \\
 		$2^{-12}$ & 0.5843 & 0.3933 & 0.2723  &	0.1907 & 0.1343 & 0.53 \\
 		\noalign{\smallskip}\hline
 	\end{tabular}
 \end{table}
\begin{table}[h!!]
	\caption{Relative error $\frac{\interleave u^{0}-u_{h}^{\rm{R}} \interleave_{\varepsilon,h}}{\interleave u^{0} \interleave_{\varepsilon,h}}$ on uniform grids in Example 3.}
	\label{tab:uniform Example3}
	\begin{tabular}{lllllll}
		\hline\noalign{\smallskip}
		$\varepsilon \setminus h$ & $2^{-2}$ & $2^{-3}$ &  $2^{-4}$  &  $2^{-5}$  & $2^{-6}$ & Rate \\
		\noalign{\smallskip}\hline\noalign{\smallskip}
		$2^{-8}$ & 0.5731 & 0.3896 & 0.2743  & 0.1992 & 0.1549 & 0.47 \\
		$2^{-10}$ & 0.5728 & 0.3886 & 0.2714  & 0.1913 & 0.1362 & 0.52 \\
		$2^{-12}$ & 0.5728 & 0.3885 & 0.2712  &	0.1908 & 0.1347 & 0.52  \\
		\noalign{\smallskip}\hline
	\end{tabular}
\end{table}

\appendix

\section{When is a locally-defined interpolation projective?}
\label{sec:appendix}

In this section, a necessary and sufficient condition is proposed for the existence of a local projection interpolation of a finite element space. As an application, the theory shows there exists no interpolations defined by local information and preserving the RRM element space.

\subsection{A sufficient and necessary condition}

{
\noindent{\bf General Assumption }
Let $\mathcal{G}_{h}$ be a subdivision of $\Omega$.  Let $W_{h}$ be a finite element space with a set of linearly independent basis $\{\phi_{k}\}_{k = 1:n}$, where $n$ denotes the dimension of $W_{h}$. Denote the support of each basis function $\phi_{k}$ by $\mathcal{M}_{k}$. Let $\big\{\mathcal{D}_{k}\big\}_{k = 1:n}$ be a set of subdomains in $\Omega$.
}

\begin{lemma}\label{lem:general no dual}
Define an interpolation $I_{h}^{\rm d}$ as: 
$
 v \mapsto I_{h}^{\rm d}v= \sum\limits_{k = 1}^{n} \lambda_{k}^{\rm d}(v)\phi_{k}(x,y),
$
 where $\lambda_{k}^{\rm d}$ is a functional associated with $\phi_{k}$, and $\lambda_{k}^{\rm d}(v)$ is computed with the information of  $v$ on $\mathcal{D}_{k}$. For each $1\leqslant k \leqslant n$, we define $\Phi_{\mathcal{D}_{k}}: =\big\{\phi_{m}|_{\mathcal{D}_{k}}: \mathcal{M}_{m}\cap \mathring{\mathcal{D}_{k}}\ne \varnothing,\ 1\leqslant m \leqslant n\big\}$ and $\Phi_{\mathcal{D}_{k}}^{*} : = \Phi_{\mathcal{D}_{k}}\backslash \big\{ \phi_{k}\big|_{\mathcal{D}_{k}}\big\}$, where $\phi_{m}|_{\mathcal{D}_{k}}$ represents the restriction of $\phi_{m}$ on $\mathcal{D}_{k}$. If there exists some $1\leqslant k_{0}\leqslant n$, such that $\phi_{k_{0}}|_{\mathcal{D}_{k_{0}}}$ can be represented linearly by these functions in $\Phi_{\mathcal{D}_{k_{0}}}^{*}$, then $I_{h}^{\rm d}$ can not preserve the space $W_{h}$.
\end{lemma}
\begin{proof}
Since $\{\phi_{k}\}_{k = 1:n}$ is a linearly independent set, $I_{h}^{\rm d}(W_{h}) = W_{h}$ is equal to the following condition: for each functional $\lambda_{k}^{\rm d}$, $1\leqslant k \leqslant n$, it holds that 
\begin{align}\label{eq: dual property}
 \lambda_{k}^{{\rm d}}(\phi_{k}) = 1 \quad \mbox{and}\quad \lambda_{k}^{{\rm d}}(\phi_{j}) = 0, \quad \forall j \ne k.
\end{align}
From the assumption, there exists a set of coefficients $\{g_{j}\}$, such that 
$$\phi_{k_{0}}|_{\mathcal{D}_{k_{0}}}  = \sum\limits_{\phi_{j}|_{\mathcal{D}_{k_{0}}} \in \Phi_{\mathcal{D}_{k}}^{*}} g_{j}\phi_{j}|_{\mathcal{D}_{k_{0}}}.
$$
If $\lambda_{k_{0}}(\phi_{k_{0}})  = \lambda_{k_{0}}(\phi_{k_{0}}|_{\mathcal{D}_{k_{0}}}) =1$, then there exists some $j_{0}\ne k_{0}$, such that $ \lambda_{k_{0}}(\phi_{j_{0}}) = \lambda_{k_{0}}(\phi_{j_{0}}|_{\mathcal{D}_{k_{0}}}) \ne 0$. It indicates that \eqref{eq: dual property} does not holds for $\lambda_{k_{0}}^{{\rm d}}$. Therefore, $I_{h}^{\rm d}$ can not preserve the space $W_{h}$.
\end{proof}
\begin{theorem}\label{thm:existence of projection}
{\rm
There exists a set of functionals $\big\{\lambda_{k}^{\rm d}\big\}_{k = 1:n}$ such that $W_h$ is contained in the domain of $\lambda_k^{\rm d}{}'s$, and $\lambda_{k}^{\rm d}(v)$ only relies on the information of $v$ on $\mathcal{D}_{k}$, with which
$
I_{h}^{\rm d}v = \sum\limits_{k = 1}^{n} \lambda_{k}^{\rm d}(v)\phi_{k}(x,y)
$ 
defines an interpolation $I_{h}^{\rm d}$ satisfying $I_{h}^{\rm d}(W_{h})=W_{h}$, if and only if, on each $\mathcal{D}_{k}$, $\phi_{k}|_{\mathcal{D}_{k}}$ can not be represented linearly by functions in $\Phi_{\mathcal{D}_{k}}^{*}$}.
\end{theorem}
\begin{proof}
The necessity is derived from Lemma~\ref{lem:general no dual}. We only have to verify the sufficiency. The construction is similar with the idea of establishing $L^{2}$ average interpolations  (see, v.g. \cite{Scott;Zhang1990}). 
For each $\mathcal{D}_{k}$, recall that $\Phi_{\mathcal{D}_{k}} =\big\{\phi_{m}|_{\mathcal{D}_{k}}: \mathcal{M}_{m}\cap \mathring{\mathcal{D}_{k}}\ne \varnothing,\ 1\leqslant m \leqslant n\big\}$. Suppose $\#(\Phi_{\mathcal{D}_{k}}) = n_{k}$, and restate $\Phi_{\mathcal{D}_{k}}$ as $\big\{\phi_{i}^{k}\big\}_{i = 1:n_{k}}$, where especially we let $\phi_{1}^{k} = \phi_{k}\big|_{\mathcal{D}_{k}}$.
Let $\{\phi_{1}^{k} , \ \phi_{2}^{k},\  \cdots, \phi_{m_{k}}^{k}\}$ be a set of maximal linearly independent groups of $\big\{\phi_{i}^{k}\big\}_{i = 1:n_{k}}$. 
We denote by $\big\{\psi_{i}^{k}\big\}_{i = 1:m_{k}}$, a $L^{2}(\mathcal{D}_{k})$-dual basis of $\big\{\phi_{i}^{k}\big\}_{k = 1:m_{k}}$. It satisfy
$
\int_{\mathcal{D}_{k}} \psi_{s}^{k}  \phi_{t}^{k}  = \delta_{st}, \ s,\ t = 1,\ \cdots, \ m_{k},
$
where $\delta_{st}$ is the Kronecker delta. 
We only concern the dual of $\phi_{1}^{k}$, therefore,
\begin{align*}
\int_{\mathcal{D}_{k}} \psi_{1}^{k}  \phi_{t}^{k}  = \delta_{1t}, \quad t = 1,\ \cdots, \ n_{k},
\end{align*}
where we utilize the fact that $\phi_{1}^{k} = \phi_{k}|_{\mathcal{D}_{k}}$ can not be represented linearly by $\big\{\phi_{i}^{k}\big\}_{i = 2:n_{k}}$, which indicates if $m_{k}<n_{k}$, then $\{\phi_{m_{k}+1}^{k} , \  \cdots, \phi_{n_{k}}^{k}\}$ can be represented linearly by $\{\phi_{2}^{k} , \ \phi_{2}^{k},\  \cdots, \phi_{m_{k}}^{k}\}$. 
Let $\psi_{k}$ be an zero extern of $\psi_{1}^{k}$ from $\mathcal{D}_{k}$ to the whole domain. Therefore we obtain a set of dual basis $\{\psi_{k}\}_{k = 1:n}$, which satisfies 
\begin{align*}
\int_{\Omega} \psi_{s}  \phi_{t}   = \int_{\mathcal{D}_{s}} \psi_{s}  \phi_{t}  =\delta_{st}, \quad s,\ t = 1,\ \cdots n.
\end{align*}
Let $\lambda_{k}^{\rm d}(v) := \int_{\Omega}\psi_{k}v {\,d x d y} $, and then $
I_{h}^{\rm d}v = \sum\limits_{k = 1}^{n} \lambda_{k}^{\rm d}(v)\phi_{k}(x,y)
$ 
defines an interpolation $I_{h}^{\rm d}$ that can preserve the space $W_{h}$. Hence the sufficiency is derived. 
\end{proof}

\subsubsection{Examples}
A very special case of this theorem is that, if every $\mathcal{D}_{k}$ is chosen as $\Omega$, then there must exist an interpolation preserve $W_{h}$, since these basis functions are linearly independent on $\Omega$. 

\paragraph{\bf Projective interpolations associated with the Crouzeix-Raviart element} Now we will take the Crouzeix-Raviart nonconforming linear element as a simple example to illustrate that, different choices of $\{\mathcal{D}_{k}\}$ in Theorem~\ref{thm:existence of projection} produce different functionals and thus different interpolations. Let $W_{h}^{\rm CR}$ be the Crouzeix-Raviart  element space defined on $\mathcal{G}_{h}$, see Figure~\ref{fig:CRmesh} ~(Left).  Let $\{\phi_{k}\}_{k=1:n}$ be set of basis of $W_{h}^{\rm CR}$. The following interpolations all satisfy $I_{h}^{\rm d}(W_{h}^{\rm CR}) = W_{h}^{\rm CR}$.
\begin{figure}
\centering
\includegraphics[height=0.35\hsize]{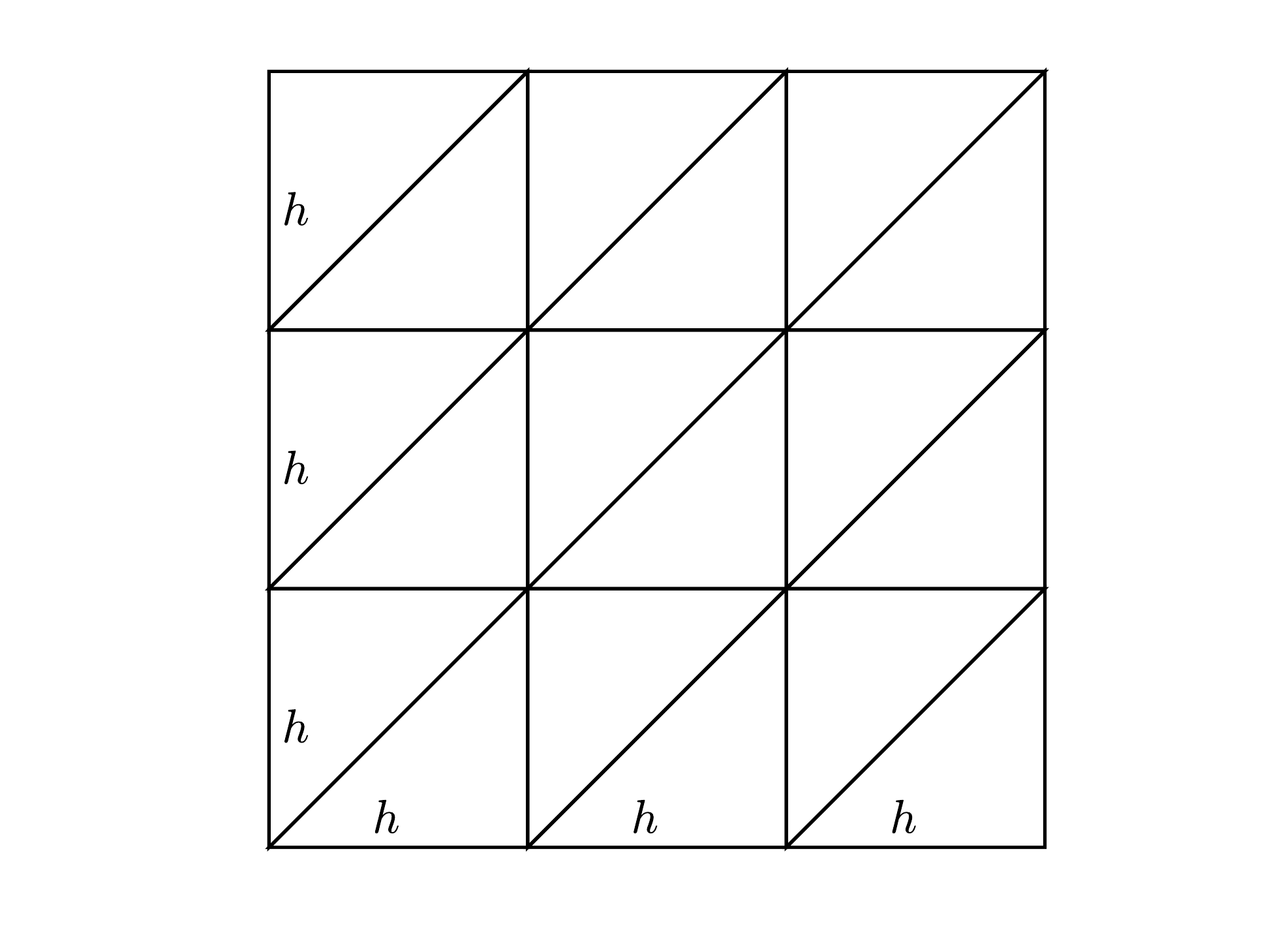}
\qquad
\includegraphics[height=0.35\hsize]{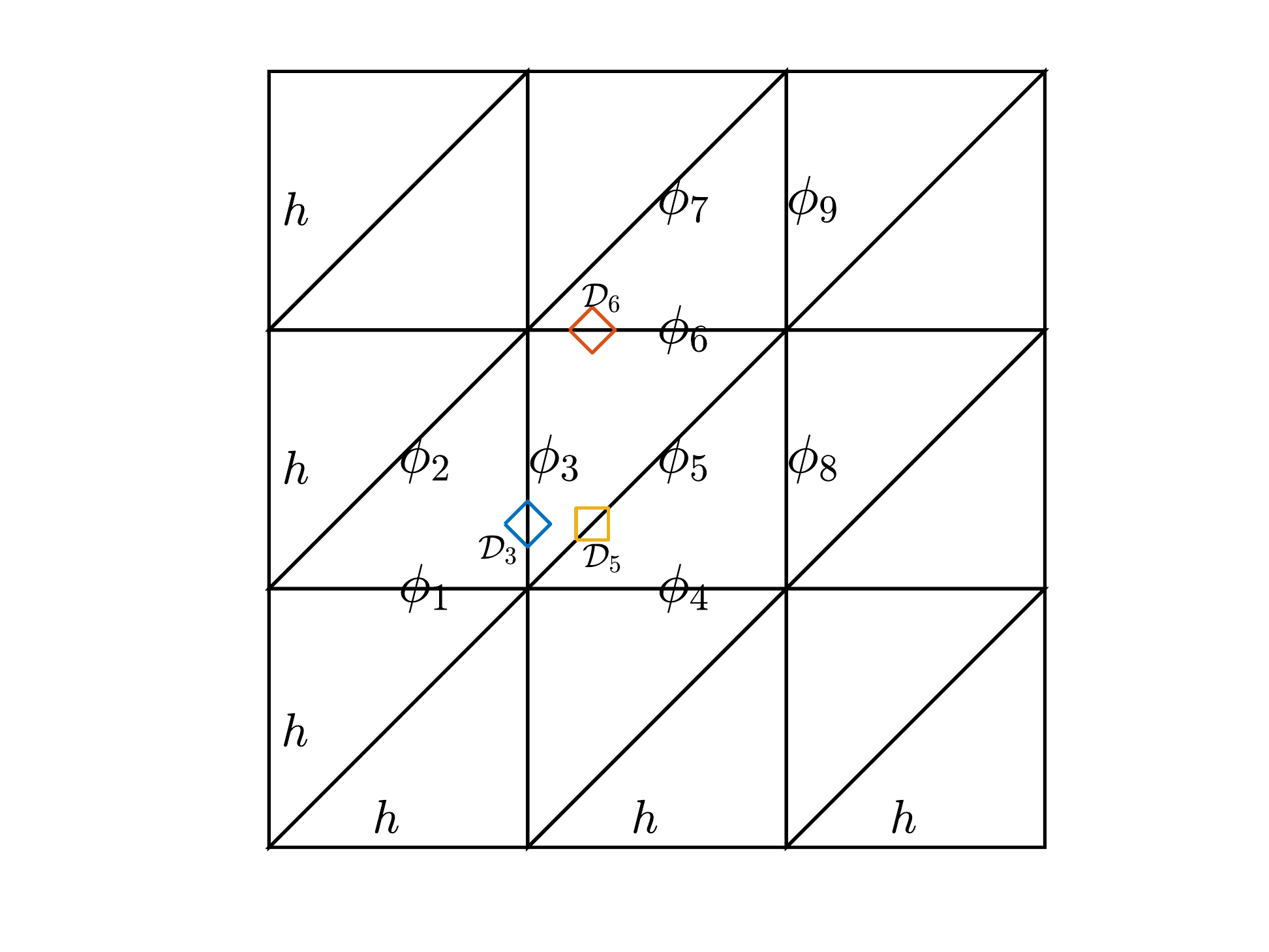}
\caption{A uniform triangular mesh (Left) and a selection of $\{\mathcal{D}_{k}\}_{k =1 :n}$ described in S3 (Right).}\label{fig:CRmesh}
\end{figure}
\begin{itemize}
\item[(S1)] For each $k$, $\mathcal{D}_{k}$ is selected as an arbitrarily small neighborhood containing each midpoint $b_{k}$ of each edge. Then we can alternatively define each functional by $\lambda_{k}(v) = v(b_{k})$. Then its corresponding interpolation is defined as 
$
I_{h}^{\rm d}(v) = \sum\limits_{k = 1}^{n}v(b_{k})\ \phi_{k},\mbox{ for any } v \in C^{0}(\Omega).
$
\item[(S2)] For each $k$, $\mathcal{D}_{k}$ is selected as an arbitrarily small neighborhood containing each edge $e_{k}$, which yields an alternatively choice of functional $\lambda_{k}(v) =  \fint_{e_{k}}v\,d s$. The corresponding interpolation is defined as 
$
I_{h}^{\rm d}(v) = \sum\limits_{k = 1}^{n} \big(\fint_{e_{k}}v\,d s\big)\ \phi_{k},$
for any
$ 
v \in H^{1}(\Omega).
$
\item[(S3)] For each $k$, $\mathcal{D}_{k}$ is selected as an arbitrarily small neighborhood containing quarter point $q_{k}$ of each edge. By Theorem~\ref{thm:existence of projection}, one can still find an $L^{2}$ average interpolation with $
\{\mathcal{D}_{k}\}_{k = 1:n}$. For instance, we choose each $\mathcal{D}_{k}$ as a rectangle with one diagonal on $e_{k}$ and the two diagonals intersecting at a quarter point of $e_{k}$; see Figure~\ref{fig:CRmesh} (Right). Let the side length of each rectangle be $\frac{h}{8\sqrt(2)}$. By the Riesz representation theorem, the $L^{2}$ dual basis $\psi_{3}$ of $\phi_{3}$, which satisfying  
$\int_{\mathcal{D}_{3}} \psi_{3}  \phi_{t}  =\delta_{3t} \ \mbox{ for any } t = 1:n$, is a linear combination of $\phi_{1}|_{\mathcal{D}_{3}}, \ \phi_{2}|_{\mathcal{D}_{3}},\ \phi_{3}|_{\mathcal{D}_{3}},\ \phi_{5}|_{\mathcal{D}_{3}},\ \phi_{6}|_{\mathcal{D}_{3}}$ when restricted on $\mathcal{D}_{3}$ and equals zero outside $\mathcal{D}_{3}$. The same is true for  $\psi_{5}, \ \psi_{6}$ and so on.
It can be computed that
\begin{align*}
& \psi_{3}|_{\mathcal{D}_{3}} = \frac{1}{h^{2}}[-18048,\ 6528,\ 12672,\ 6528,\ 31104]\cdot[\phi_{1}, \ \phi_{2},\ \phi_{3},\ \phi_{5},\ \phi_{6}]^{t}\big|_{\mathcal{D}_{3}}, \\
& \psi_{5}|_{\mathcal{D}_{5}} = \frac{1}{h^{2}}[q_{1},\ q_{2},\ 24960,\ q_{2},\ q_{1}]\cdot[\phi_{3}, \ \phi_{6},\ \phi_{5},\ \phi_{8},\ \phi_{4}]^{t}\big|_{\mathcal{D}_{5}}, \\
& \psi_{6}|_{\mathcal{D}_{6}} = \frac{1}{h^{2}}[-18048,\ 6528,\ 12672,\ 6528,\ 31104]\cdot[\phi_{3}, \ \phi_{5},\ \phi_{6},\ \phi_{7},\ \phi_{9}]^{t}\big|_{\mathcal{D}_{6}}
\end{align*}
where $ q_{1} = -384(8\sqrt{2} - 1), \ q_{2} = -384(8\sqrt{2} - 129)$.
The corresponding interpolation is defined as 
$
I_{h}^{\rm d}(v) = \sum\limits_{k = 1}^{n} \big(\int_{\mathcal{D}_{k}}\psi_{k} v {\,d x d y} \big)\ \phi_{k},\mbox{ for any } v \in L^{2}(\Omega).
$
\end{itemize}

\subsection{Non-existence of locally-defined projection associated with RRM element}
Next we consider the RRM element space $V_{h0}^{\rm R}$. Denote $\big\{\varphi_{K}\big\}_{K\in\mathcal{K}_{h}^{i}}$ as a linearly independent basis in $V_{h0}^{\rm R}$. 
\begin{lemma}\label{lem:linear dependent omega}
A subdomain $\omega$ is named as a completely subdomain of $\Omega$, if it is a union of elements in $\mathcal{G}_{h}$ and each element in $\omega$ is located in the supports of nine basis functions in $V_{h0}^{\rm R}$. Denote $\Phi_{\omega} : = \big\{ \varphi_{T}|_{\omega} : \mathcal{M}_{T}\cap \mathring{\omega} \ne \varnothing,\ T\in \mathcal{K}_{h}^{i}\big\}$. Then with a set of nonzero coefficients $\{d_{T} \in \mathcal{C}_{h}:\ T \subset \omega\}$, it holds that
\begin{align*}
\sum\limits_{\varphi_{T}|_{\omega}  \in \Phi_{\omega}} d_{T} L_{T}H_{T}\varphi_{T}|_{\omega} = 0.
\end{align*}
\end{lemma}
\begin{proof}
By Lemma~\ref{lem:property of phi}~$(c)$, there exists a checkerboard coefficients set $\mathcal{C}_{h} = \{d_{T}:\ T\in \mathcal{J}_{h}\}$, such that, for $(x,y) \in \omega$,
\begin{align*}
\sum\limits_{T\in \mathcal{J}_{h}} d_{T} L_{T}H_{T}\varphi_{T}(x,y) = \sum\limits_{\varphi_{T}|_{\omega}  \in \Phi_{\omega}} d_{T} L_{T}H_{T}\varphi_{T}(x,y) = 0,
\end{align*}
which yields that
$
\sum\limits_{\varphi_{T}|_{\omega}  \in \Phi_{\omega}} d_{T} L_{T}H_{T}\varphi_{T}|_{\omega} = 0.
$
\end{proof}
\begin{theorem}\label{thm:no dual basis RRM}
{Let $\{\mathcal{D}_{K}\}_{K \in \mathcal{K}_{h}^{i}}$ be a set of subdomains of $\Omega$. Let $\big\{\lambda_{K}^{\rm d}\big\}$ be a set of functionals, satisfying that $\lambda_{K}^{\rm d}(v)$ is computed with the information of $v$ within $\mathcal{D}_{K}$. Define an interpolation operator as
\begin{align*}
\Pi_{h}^{{\rm d}} : v \mapsto \Pi_{h}^{{\rm d}}v= \sum\limits_{K\in \mathcal{K}_{h}^{i}} \lambda_{K}^{{\rm d}}(v)\varphi_{K}(x,y).
\end{align*} 
If there exists some $\mathcal{D}_{K_{0}}$ that is a completely subdomain described in Lemma~\ref{lem:linear dependent omega}, then $\Pi_{h}^{\rm d}$ can not be a projection which satisfies $\Pi_{h}^{{\rm d}}(V_{h0}^{\rm R}) = V_{h0}^{\rm R}$. 
}
\end{theorem}
\begin{proof} 
From the assumption, $\mathcal{D}_{K_{0}}$ is a completely subdomain of $\Omega$. By Lemma~\ref{lem:linear dependent omega}, it holds that
\begin{align}\label{eq:linear dependent}
\sum\limits_{\varphi_{T}|_{\mathcal{D}_{K_{0}}}  \in \Phi_{\mathcal{D}_{K_{0}}}} d_{T} L_{T}H_{T}\varphi_{T}|_{\mathcal{D}_{K_{0}}} = 0.
\end{align}
Let $\Phi_{\mathcal{D}_{K_{0}}}^{*} = \Phi_{\mathcal{D}_{K_{0}}} \backslash (\varphi_{K_{0}}|_{\mathcal{D}_{K_{0}}})$.  Since $d_{K_{0}}\ne 0$, we obtain  
\begin{align*}
\varphi_{K_{0}}\big|_{ \mathcal{D}_{K_{0}} }  = \sum\limits_{\varphi_{T}|_{\mathcal{D}_{K_{0}}} \in \Phi_{\mathcal{D}_{K_{0}}}^{*}} g_{T}\varphi_{T}|_{ \mathcal{D}_{K_{0}}},
\end{align*}
where $\{g_{T}\}$ is derived by~\eqref{eq:linear dependent}. 
Therefore, from~Theorem~\ref{thm:existence of projection}, $\Pi_{h}^{\rm d}$ can not preserve $V_{h0}^{\rm R}$. 
\end{proof}
For the RRM element space, Theorem \ref{thm:no dual basis RRM} reveals that, there can be no interpolation preserving the space $V_{h0}^{\rm R}$ associated with the local basis functions, unless all functionals are computed with the global information of $v$.


\begin{thebibliography}{10}

\bibitem{Brenner;Scott2007}
S.~C. Brenner and R.~Scott.
\newblock {\em The mathematical theory of finite element methods}, volume~15.
\newblock Springer Science \& Business Media, 2007.

\bibitem{Brenner;Neilan2011}
S.~C. Brenner and M.~Neilan.
\newblock A {$C^{0}$} interior penalty method for a fourth order elliptic
  singular perturbation problem.
\newblock {\em SIAM Journal on Numerical Analysis}, 49(2):869--892, 2011.

\bibitem{Chen;Chen2014}
H.~Chen and S.~Chen.
\newblock Uniformly convergent nonconforming element for {3-D} fourth order
  elliptic singular perturbation problem.
\newblock {\em J. Comput. Math}, 32(6):687--695, 2014.

\bibitem{Chen;Chen;Qiao2013}
H.~Chen, S.~Chen, and Z.~Qiao.
\newblock {$C^{0}$}-nonconforming tetrahedral and cuboid elements for the
  three-dimensional fourth order elliptic problem.
\newblock {\em Numerische Mathematik}, 124(1):99--119, 2013.

\bibitem{Chen;Chen;Xiao2014}
H.~Chen, S.~Chen, and L.~Xiao.
\newblock Uniformly convergent {$C^{0}$}-nonconforming triangular prism element
  for fourth-order elliptic singular perturbation problem.
\newblock {\em Numerical Methods for Partial Differential Equations},
  30(6):1785--1796, 2014.

\bibitem{Chen;Liu;Qiao2010}
S.~Chen, M.~Liu, and Z.~Qiao.
\newblock An anisotropic nonconforming element for fourth order elliptic
  singular perturbation problem.
\newblock {\em International Journal of Numerical Analysis {$\&$} Modeling},
  7(4), 2010.

\bibitem{Chen;Zhao;Shi2005}
S.~Chen, Y.~Zhao, and D.~Shi.
\newblock Non {$C^{0}$} nonconforming elements for elliptic fourth order
  singular perturbation problem.
\newblock {\em Journal of Computational Mathematics}, pages 185--198, 2005.

\bibitem{P.G.Ciarlet2002}
P.~G. Ciarlet.
\newblock {\em The finite element method for elliptic problems}, volume~40.
\newblock North-Holland Publishing Company, 2002.

\bibitem{Clement1975}
P.~Cl{\'e}ment.
\newblock Approximation by finite element functions using local regularization.
\newblock {\em Revue fran{\c{c}}aise d'automatique, informatique, recherche
  op{\'e}rationnelle. Analyse num{\'e}rique}, 9(R2):77--84, 1975.

\bibitem{Demko1985}
S.~Demko.
\newblock On the existence of interpolating projections onto spline spaces.
\newblock {\em Journal of Approximation Theory}, 43(2):151--156, 1985.

\bibitem{Dupont;Scott1980}
T.~Dupont and R.~Scott.
\newblock Polynomial approximation of functions in {Sobolev} spaces.
\newblock {\em Mathematics of Computation}, 34(150):441--463, 1980.

\bibitem{Fichera2006}
G.~Fichera.
\newblock {\em Linear Elliptic Differential Systems and Eigenvalue Problems},
  volume~8.
\newblock Springer, 2006.

\bibitem{Fortin.M;Soulie.M1983}
M.~Fortin and M.~Soulie.
\newblock A non-conforming piecewise quadratic finite element on triangles.
\newblock {\em International Journal for Numerical Methods in Engineering},
  19(4):505--520, 1983.

\bibitem{Frank1997}
L.~S. Frank.
\newblock {\em Singular perturbations in elasticity theory}, volume~1.
\newblock IOS Press, 1997.

\bibitem{Franz;Roos;Wachtel2014}
S.~Franz, H.-G. Roos, and A.~Wachtel.
\newblock A {$C^{0}$} interior penalty method for a singularly-perturbed
  fourth-order elliptic problem on a layer-adapted mesh.
\newblock {\em Numerical Methods for Partial Differential Equations},
  30(3):838--861, 2014.

\bibitem{Guzman;Leykekhman;Neilan2012}
J.~Guzm{\'a}n, D.~Leykekhman, and M.~Neilan.
\newblock A family of non-conforming elements and the analysis of {Nitsche’s}
  method for a singularly perturbed fourth order problem.
\newblock {\em Calcolo}, 49(2):95--125, 2012.

\bibitem{XY.Meng;XQ.Yang;S.Zhang2016}
X.~Meng, X.~Yang, and S.~Zhang.
\newblock Convergence analysis of the rectangular {Morley} element scheme for
  second order problem in arbitrary dimensions.
\newblock {\em Science China Mathematics}, 59(11):2245--2264, 2016.

\bibitem{Nilssen;Tai;Winther2001}
T.~Nilssen, X.~Tai, and R.~Winther.
\newblock A robust nonconforming ${H}^2$-element.
\newblock {\em Mathematics of Computation}, 70(234):489--505, 2001.

\bibitem{Park.C;Sheen.D2003}
C.~Park and D.~Sheen.
\newblock {$P_{1}$}-nonconforming quadrilateral finite element methods for
  second-order elliptic problems.
\newblock {\em SIAM Journal on Numerical Analysis}, 41(2):624--640, 2003.

\bibitem{2Sablonniere.P2003}
P.~Sablonni{\`e}re.
\newblock On some multivariate quadratic spline quasi-interpolants on bounded
  domains.
\newblock In {\em Modern developments in multivariate approximation}, pages
  263--278. Springer, 2003.

\bibitem{Sablonniere.P2003}
P.~Sablonni{\`e}re.
\newblock Quadratic spline quasi-interpolants on bounded domains of
  {$\mathbb{R}^{d}$}, d= 1, 2, 3.
\newblock {\em Rend. Sem. Mat. Univ. Pol. Torino}, 61(3):229--246, 2003.

\bibitem{Grisvard1985}
M.~Schechter.
\newblock Elliptic problems in nonsmooth domains {(P. Grisvard)}.
\newblock {\em SIAM Review}, 28(1):125--125, 1986.

\bibitem{Schumaker2007}
L.~Schumaker.
\newblock {\em Spline functions: basic theory}.
\newblock Cambridge University Press, 2007.

\bibitem{Scott;Zhang1990}
L.~R. Scott and S.~Zhang.
\newblock Finite element interpolation of nonsmooth functions satisfying
  boundary conditions.
\newblock {\em Mathematics of Computation}, 54(190):483--493,
  04 1990.

\bibitem{Wang.M;Shi.Z2013mono}
Z.~Shi and M.~Wang.
\newblock {\em Finite element methods}.
\newblock Science Press, Beijing, 2013.

\bibitem{Tai;Winther2006}
X.~Tai and R.~Winther.
\newblock A discrete de {Rham} complex with enhanced smoothness.
\newblock {\em Calcolo}, 43(4):287--306, 2006.

\bibitem{Wang;Wu;Xie2013}
L.~Wang, Y.~Wu, and X.~Xie.
\newblock Uniformly stable rectangular elements for fourth order elliptic
  singular perturbation problems.
\newblock {\em Numerical Methods for Partial Differential Equations},
  29(3):721--737, 2013.

\bibitem{Wang2001}
M.~Wang.
\newblock On the necessity and sufficiency of the patch test for convergence of
  nonconforming finite elements.
\newblock {\em SIAM journal on numerical analysis}, 39(2):363--384, 2001.

\bibitem{Wang;Meng2007}
M.~Wang and X.~Meng.
\newblock A robust finite element method for a {3-D} elliptic singular
  perturbation problem.
\newblock {\em Journal of Computational Mathematics}, pages 631--644, 2007.

\bibitem{Wang;Shi;Xu2007}
M.~Wang, Z.~Shi, and J.~Xu.
\newblock A new class of {Zienkiewicz}-type non-conforming element in any
  dimensions.
\newblock {\em Numerische Mathematik}, 106(2):335--347, 2007.

\bibitem{2ndWang;Shi;Xu2007}
M.~Wang, Z.~Shi, and J.~Xu.
\newblock Some {$n$}-rectangle nonconforming elements for fourth order elliptic
  equations.
\newblock {\em Journal of Computational Mathematics}, pages 408--420, 2007.

\bibitem{Wang;Xu;Hu2006}
M.~Wang, J.~Xu, and Y.~Hu.
\newblock Modified {Morley} element method for a fourth order elliptic singular
  perturbation problem.
\newblock {\em Journal of Computational Mathematics}, pages 113--120, 2006.

\bibitem{WangRenhong2013}
R.~Wang.
\newblock {\em Multivariate spline functions and their applications}, volume
  529.
\newblock Springer Science \& Business Media, 2013.

\bibitem{Wang;Lu1998}
R.~Wang and Y.~Lu.
\newblock Quasi-interpolating operators and their applications in hypersingular
  integrals.
\newblock {\em Journal of Computational Mathematics}, pages 337--344, 1998.

\bibitem{Wang;Huang;Tang;Zhou2018}
W.~Wang, X.~Huang, K.~Tang, and R.~Zhou.
\newblock {Morley-Wang-Xu} element methods with penalty for a fourth order
  elliptic singular perturbation problem.
\newblock {\em Advances in Computational Mathematics}, 44(4):1041--1061, 2018.

\bibitem{Xie;Shi;Li2010}
P.~Xie, D.~Shi, and H.~Li.
\newblock A new robust {$C^{0}$}-type nonconforming triangular element for
  singular perturbation problems.
\newblock {\em Applied Mathematics and Computation}, 217(8):3832--3843, 2010.

\bibitem{J.Xu1992}
J.~Xu.
\newblock Iterative methods by space decomposition and subspace correction.
\newblock {\em SIAM Review}, 34(4):581--613, 1992.

\bibitem{Zeng.H;Zhang.C;Zhang.S2019}
H.~Zeng, C.~Zhang, and S.~Zhang.
\newblock {Optimal quadratic element on rectangular grids for {$H^{1}$}
  problems}.
\newblock {\em {BIT} Numerical Mathematics}, accepted.

\bibitem{Zhang.S2016nm}
S.~Zhang.
\newblock Stable finite element pair for {Stokes} problem and discrete stokes
  complex on quadrilateral grids.
\newblock {\em Numerische Mathematik}, 133(2):371--408, 2016.

\bibitem{Shuo.Zhang2020}
S.~Zhang.
\newblock {Minimal consistent finite element space for the biharmonic equation
  on quadrilateral grids}.
\newblock {\em IMA Journal of Numerical Analysis}, 40(2):1390--1406, 2020.

\bibitem{Zhang;Wang2008}
S.~Zhang and M.~Wang.
\newblock A posteriori estimator of nonconforming finite element method for
  fourth order elliptic perturbation problems.
\newblock {\em Journal of Computational Mathematics}, pages 554--577, 2008.

\end{thebibliography}
\end{document}